\documentclass{amsart}
\usepackage{amsmath,amssymb,amsthm,graphicx,epsfig,float,url}
\usepackage[colorlinks=true,citecolor={Plum},linkcolor={Periwinkle}]{hyperref}
\usepackage{pdfsync}
\usepackage[usenames, dvipsnames]{xcolor}
\usepackage{tikz}
\usepackage{subfig}
\usepackage{bbm}
\usepackage{stmaryrd}
\usepackage{amssymb}
\usepackage{mathrsfs}  
\usepackage{caption}
\usepackage{float}
\usepackage{esint}
\usepackage{verbatim}
\usepackage{geometry}
\usepackage{comment}

\usepackage{mathtools}

\usepackage{dsfont}
\usepackage[makeroom]{cancel}

\usepackage{todonotes}

\usepackage{marginnote}
\newcommand*\mnote[3][0pt]{%
  \if l#2\reversemarginpar\def\pointer{\filledmedtriangleright}%
  \def\stackalignment{r}\fi%
  \if r#2\normalmarginpar\def\pointer{\filledmedtriangleleft}%
  \def\stackalignment{l}\fi%
  \marginpar{%
    \topinset{%
      \scalebox{1.5}{\textcolor{magenta}{$\pointer$}}}{%
      \belowbaseline[-1.5\baselineskip-#1]{%
        \stackengine%
        {-5pt}%
        {\fcolorbox{magenta}{white}{\parbox{1.8cm}%
            {\vspace{3pt}\raggedright#3}}}%
        {~\colorbox{white}{\sffamily Nota}}%
        {O}%
        {l}%
        {F}%
        {F}%
        {S}%
      }%
    }{%
      3ex+#1}{%
      -2ex}%
  }%
}

\definecolor{myblue}{RGB}{89, 128, 212}

\newcommand{\BBB}{\color{black}}

\newcommand{\R}{\textnormal{I\kern-0.21emR}}
\newcommand{\N}{\textnormal{I\kern-0.21emN}}

\newcommand{\C}{\mathscr{C}}

\newcommand{\Lm}{\mathcal{L}}
 
\renewcommand{\geq}{\geqslant}
\renewcommand{\leq}{\leqslant}

\def\B{{\mathbb B}}

\def\e{{\varepsilon}}

\def\th{{\theta}}

\allowdisplaybreaks

\def\YYint#1#2#3{{\setbox0=\hbox{$#1{#2#3}{\iint}$}
    \vcenter{\hbox{$#2#3$}}\kern-.51\wd0}}
 
\usepackage[dvipsnames]{xcolor}

\newtheorem*{theorem*}{Theorem}

\newtheorem{theorem}{Theorem}

\newtheorem{material}{material}[section]
\newtheorem{proposition}[material]{Proposition}
\newtheorem{corollary}[material]{Corollary}
\newtheorem{definition}[material]{Definition}
\newtheorem{lemma}[material]{Lemma}

\newtheorem{remark}[material]{Remark}
\newtheorem{theoremnonumbering}{Theorem}

\def\O{{\Omega}}

\def\n{{\nabla}}

\def\p{{\varphi}}

\newcommand{\abs}[1]{{\left|#1\right|}}
\newcommand{\ov}[1]{{\overline{#1}}}
\newcommand{\norma}[1]{{\left\Vert#1\right\Vert}}

\newcommand{\eps}{\varepsilon}
\newcommand{\Om}{\Omega}

\newcommand{\nabsig}{\nabla_\Sigma}
\newcommand{\nabsigt}{\nabla_{\Sigma_t}}
\newcommand{\lamti}{\tilde{\lambda}}
\newcommand{\lambar}{\overline{\lambda}}
\newcommand{\parnu}{\partial_\nu}
\newcommand{\parnunu}{\partial_{\nu\nu}}
\newcommand{\parnut}{\partial_{\nu_t}}
\newcommand{\parnunut}{\partial_{\nu_t\nu_t}}

\renewcommand{\S}{{\mathbb{S}^{d-1}}}
\newcommand{\Ha}{\mathcal{H}}

\usepackage{xargs}

\numberwithin{equation}{section}

\title[Optimisation problems for bulk-surface systems]{Some optimal control and shape optimisation problems for bulk-surface cooperative systems }
\author{Andrea Gentile}
\address{Mathematical and Physical Sciences for Advanced Materials and Technologies, Scuola Superiore Meridionale, Largo San Marcellino 10, Napoli 80126, Italy.}\email{andrea.gentile2@unina.it}
\author{Idriss Mazari-Fouquer}
\address{CEREMADE, UMR CNRS 7534, Universit\'e Paris-Dauphine, Universit\'e PSL, Place du Mar\'echal De Lattre De Tassigny, 75775 Paris cedex 16, France.}
\email{mazari@ceremade.dauphine.fr}
\author{Rapha\"{e}l Prunier}
\address{CEREMADE, UMR CNRS 7534, Universit\'e Paris-Dauphine, Universit\'e PSL, Place du Mar\'echal De Lattre De Tassigny, 75775 Paris cedex 16, France.}
\email{prunier@ceremade.dauphine.fr}

\begin{document}

\maketitle

\begin{abstract}
The goal of this paper is to address some optimal control and shape optimisation problems arising from bulk-surface cooperative systems. The basic model under consideration is the following: letting $\O$ be a fixed domain, we assume that a population (with density $u$) lives inside $\O$ and can access some resources $f$, while a second population (with density $v$) lives on the boundary $\partial \O$ and can access other resources $g$. These two populations are coupled in a cooperative manner by a constant exchange rate at the boundary, leading to a non-standard PDE system that has already been studied \cite{BGT} for its connection with road-field models. Building on the considerations of \cite{BGT}, we have two main objectives here: first, investigate the question of \emph{optimal resources distribution} inside the domain $\O$ and on the surface $\partial \O$, \emph{i.e.} how to spread resources in order to guarantee an optimal survival of the two species. We establish rigid Talenti inequalities and comparison results when $\O$ is a ball, extending in particular the results of J. J. Langford \cite{Langford_Neumann,Lanford_Robin}. Second, when the resources distribution $f$ and $g$ are constant, we provide a partial analysis of the natural shape optimisation problem: \textit{which shape $\O$ maximises the survival rate of the two species}? Namely, we show that in certain regimes there can be no optimal shape and, by computing second-order shape derivatives, we investigate the local optimality of the ball.
\end{abstract}


\textbf{Acknowledgement:} This work was started during a visit of A. Gentile to I. Mazari-Fouquer within the framework of the doctoral program of the Scuola Normale Meriodionale. The second and third author were supported by a PSL Young Researcher Starting Grant 2023 (P.I.: I. Mazari-Fouquer) from Paris Dauphine Universit\'e PSL.

\section{Introduction}

\subsection{Scope of the paper}

The main model under study in the present article is a \emph{cooperative} system of elliptic partial differential equations (PDE) coupled \emph{via} an exchange term. To be more specific, consider a smooth, bounded domain $\O\subset \R^d$ with boundary $\Sigma:=\partial \O$ and two functions $f:\O\to \R$ and $g:\Sigma\to \R$. We are interested in the behaviour of two populations $\theta_\O$ and $\theta_\Sigma$, living respectively inside $\O$ and on $\Sigma$. The first population can access resources distributed according to $f$, while the second one can access resources distributed according to $g$. In each case, this leads to a linear growth term $f\theta_\O$ or $g\theta_\Sigma$. These population are respectively subject to a Malthusian (quadratic) death term $-\theta_{\O}^2$ and $-\theta_{\Sigma}^2$, which accounts for crowding effects. Finally, there is an exchange of the two populations on the boundary: a portion $\kappa\theta_\O$ of the first population passes from the interior to the boundary of the domain, while a portion $\kappa \theta_\Sigma$ of the second passes from the boundary to the interior. For notational simplicity, we take $\kappa=1$, so that the overall system reads
\begin{equation}\label{Eq:basic}
\begin{cases}
-\Delta \theta_\O -\theta_\O \left(f-\theta_\O\right)=0&\text{ in }\O\,, 
\\ \partial_{\nu_\Sigma} \theta_\O+ \theta_\O=\theta_\Sigma&\text{ on }\Sigma\,, 
\\ -\Delta_\Sigma \theta_\Sigma-\theta_\Sigma(g-\theta_\Sigma)=\theta_\O-\theta_\Sigma\,&\text{ on }\Sigma\,, 
\\ \theta_\O\,, \theta_\Sigma>0,&\text{ in }\O\,,
\end{cases}\end{equation}
 where $\nu=\nu_\Sigma$ denotes the outer unit normal to $\Sigma$, and $-\Delta_\Sigma$ stands for the tangential Laplacian over $\Sigma$.

\textbf{Modelling considerations}

Equation \eqref{Eq:basic} appears in a variety of contexts; while we defer a more detailed discussion of the literature to section \ref{Se:Bib} below, let us point out that from the interpretative point of view, one might consider the case where both $\O$ and $\Sigma$ are unbounded, in which case $\O$ can be dubbed a \emph{field}, while $\Sigma$ can be called a \emph{road}. This type of field-road models can be used to model the effects of transport networks on the propagation of biological species or diseases, as in   \cite{berestycki_fisherkpp_2013, berestycki_influence_2013,berestycki_shape_2016}. In the specific case under consideration here, another application comes from cell-polarisation or cellular division (see \cite{Fellner_Latos_Tang,Morita_Sakamoto, Ratz_Roger}), where $\O$ can be thought of as \emph{bulk}, while $\Sigma$ is interpreted as a \emph{surface}. In both cases, the way the geometry of $\O$ impacts the dynamics of the parabolic system associated with \eqref{Eq:basic} is a question of interest. A major motivation of the present article is the work of Bogosel, Giletti \& Tellini in \cite{BGT}, which focuses on the road-field interpretation of the model and seeks to optimise the invasion speed of the species. This problem boils down to a family of spectral optimisation problems that we also tackle here with a different outlook.

\textbf{Mathematical structure of the equation}

Equation \eqref{Eq:basic} is a Fisher-KPP model which belongs to the class of \emph{monostable} elliptic PDEs. In the scalar case (discarding the equation on $\theta_\Sigma$) this equation has been instrumental in population dynamics, and we refer to the monographs \cite{zbMATH02027761,zbMATH07668634}. In the case of systems, the analytical properties of \eqref{Eq:basic} are also well-understood. This is due to the \emph{cooperative} nature of the system, by which we mean the following: rewriting \eqref{Eq:basic} under the form 
\[ \begin{cases} -\Delta \theta_\O=F(x,\theta_\O)-\theta_\O  \Ha^{d-1}_\Sigma+\theta_\Sigma\Ha^{d-1}_\Sigma=F_1(x,\theta_\O)+G_1(\theta_\Sigma)\,, 
\\ -\Delta_\Sigma \theta_\Sigma=F_2(x,\theta_\Sigma)+\theta_\O=F_2(x,\theta_\Sigma)+G_2(\theta_\O)\,,\end{cases}\]  for some functions $F_1$,$F_2$ and $G_1$, $G_2$,  where $\Ha^{d-1}_\Sigma$ is the $(d-1)$ dimensional measure on $\Sigma$, we have $\partial_{\theta}G_k(\theta)>0$, $k=1,2$. Mathematically, this implies a strong maximum principle for the resulting system (see \cite{Girardin_Mazari}). This maximum principle, and the particular form of the non-linearities make it an easy consequence (see \cite{zbMATH02194918,zbMATH07668634}) that the existence, uniqueness and dynamical stability  of a solution to \eqref{Eq:basic} is equivalent to requiring that the  associated first eigenvalue $\Lambda_\O(f,g)$ of the problem
\begin{equation}\label{eq:egvPb}\begin{cases}
-\Delta u-fu=\Lambda_\O(f,g)u&\text{ in }\O\,, \\
\partial_\nu u+ u=v&\text{ on }\Sigma\,, 
\\ -\Delta_\Sigma v+(1-g)v=\Lambda_\O(f,g)v+u&\text{ on }\Sigma,\\
 u>0 \text{ in }\O, v>0 \text{ on } \Sigma,\end{cases}
\end{equation}
be negative: 
\[ \Lambda_\O(f,g)<0.\] Alternatively, this  principal eigenvalue can be defined variationally by
\begin{equation*}\label{Eq:Eigenvalue} \Lambda_\O(f,g):=\min_{\substack{(u,v)\in W^{1,2}(\O)\times W^{1,2}(\Sigma) \\ (u,v)\ne(0,0)}}\frac{\int_\O |\n u|^2+\int_\Sigma |\nabsig v|^2-\int_\O fu^2-\int_\Sigma gv^2+\int_\Sigma (u-v)^2}{\int_\O u^2+\int_\Sigma v^2}\end{equation*}
 where  $W^{1,2}(\O)$ and $ W^{1,2}(\Sigma)$ denote  the standard Sobolev spaces in $\O$ and on $\Sigma$ and $\nabsig$  is the tangential gradient (note that $\Lambda_\Om(f,g)$ is well-defined through \eqref{eq:egvPb} as soon as $\Om$ is Lipschitz).
As such, and as is customary in population dynamics (see \cite{zbMATH07573593}), it is natural to look for the  triplet $(\O,f,g)$ (under some natural constraints) that minimises $\Lambda_\O(f,g)$. There are two natural sub-problems, which are the main focus of this paper:
\begin{itemize}
\item \underline{The optimal control problem.} In this first problem, the domain $\O$ is fixed, and we look for the optimal couple $(f,g)$ to minimise $\Lambda_\O$. We will focus on the case where $\O=\B$ is the ball of radius $1$ and investigate symmetry properties of $(f,g)$. As we shall see, it is an immediate consequence of classical rearrangement inequalities that $f$ and $g$ should be as symmetric as possible (see Proposition \ref{Pr:FK}). Our point of view, however, will be to establish broader classes of Talenti inequalities for such coupled systems, thus providing a finer comparison between the solutions of a linear PDE and of a symmetrised counterpart. These results involve the so-called \textit{cap symmetrisation} (see Definition \ref{De:CapSymmetrisation} below), which consists in a separate symmetrisation on each sphere, with the aim of comparing a solution of the coupled system with its symmetrised version. 
From a methodological point of view, we build on fundamental contributions of J. J. Langford in \cite{Langford_Neumann,Lanford_Robin} (which we recall in Theorem \ref{Th:Langford}) to derive appropriate rigid Talenti inequalities (in Theorem \ref{theo:rigidity}). Observe that we also derive along the way rigidity results for the Talenti inequalities obtained by J. J. Langford (see Theorem \ref{theo:rigidity1}). 
\item \underline{The shape optimisation problem.} In this second problem, which was already the focus of \cite{BGT},  $f$ and $g$ are fixed, and we seek to minimise $\Lambda_\O$ with respect to $\O$ under the volume constraint $|\O|\leq |\B|$. While we are not able to provide a conclusive answer to this question, when $f$ and $g$ are constant we identify regimes where this optimisation problem has no solution (in Theorem \ref{Th:NonExistence}) and we investigate the first and second-order optimality conditions at the ball (in Theorem \ref{thm:local_ball}). As this latter analysis relies on second-order shape derivatives, we thereby extend the results from \cite{BGT}.
\end{itemize}

\subsection{Main results  and methods }

Throughout, $\B\subset\R^d$  denotes the unit ball and $\S:=\partial\B$ the unit sphere.

\subsubsection{The optimal control problem: Talenti inequalities} 
We first recall the definition of cap symmetrisation. For any $r\in(0,1]$, ${\mathbb S^{d-1}_r}$ denotes the sphere of radius $r$ endowed with the natural spherical distance $d_{\mathbb S^{d-1}_r}$. Letting $\{e_i\}_{i\in \{1,\dots,d\}}$ be the canonical basis of $\R^d$, we use the notation $\theta_{\max}^r:=d_{\mathbb S^{d-1}_r}(re_1,-re_1)$.  The notation $\mathcal{H}^{d-1}$ stands for the Hausdorff measure of dimension $d-1$.
\begin{definition}\label{De:CapSymmetrisation}
For $r\in(0,1]$ and $\theta\in\left( 0,\th_{\max}^r\right)$ we define the spherical cap  $K_r(\theta):=\{x\in \mathbb{S}^{d-1}_{r}, d_{\mathbb S^{d-1}_r}(x,re_1)<\theta\}$. 
\begin{itemize}
\item Let $f\in L^1({\mathbb S^{d-1}_r})$ with $f\geq0$. There exists a unique (up to modification on a set with zero $\mathcal{H}^{d-1}$ measure) \BBB function $f^{\sharp}:{\mathbb S^{d-1}_r}\to \R$ which has the following properties:
\begin{enumerate}
\item $f^{\sharp}$ is constant on  $\partial K_r(\theta)$ for any  $\theta\in (0,\th_{\max}^r)$ and letting, with a slight abuse of notation, $f^{\sharp}(\theta)$ be this common value, the map $\theta\mapsto f^{\sharp}(\theta)$ is non-increasing.
\item $f^{\sharp}$ and $f$ have the same distribution function: it holds
\[ \forall t \in \R\,, \mathcal H^{d-1}\left(\{x\in {\mathbb S^{d-1}_r}\,, f(x)\geq t\}\right)= \mathcal H^{d-1}\left(\{x\in {\mathbb S^{d-1}_r}\,, f^{\sharp}(x)\geq t\}\right).\]
\end{enumerate}
This function is called the \emph{cap symmetrisation} of the function $f$.
\item
For any function $f\in L^1(\B)$ with $f\geq0$,  as $f|_{\mathbb S_r}\in L^1(\mathbb S^{d-1}_r)$ for a.e. $r\in(0,1]$, \BBB we can define its cap symmetrisation (with a slight abuse of notation) $f^{\sharp}:\B\to \R$ as follows: for  a.e. \BBB $r\in (0,1]$, $f^{\sharp}|_{\mathbb S_r}:{\mathbb S^{d-1}_r}\to \R$ is the cap symmetrisation of $f|_{\mathbb S_r}:{\mathbb S^{d-1}_r}\to \R$.
\item
Finally, for any $f\in L^1(\B)$ with $f\geq0$, $f_{\sharp}$ denotes its decreasing cap symmetrisation; it is defined similarly, the main difference being that the map $\theta\mapsto f_{\sharp}(\theta)$ is required to be non-decreasing.
\end{itemize}
\end{definition}
As a consequence of general rearrangement inequalities that will be recalled in due time, we obtain the following symmetrisation result for $\Lambda_\O(f,g)$ when $\O=\B$ is the unit ball.
\begin{proposition}\label{Pr:FK}
 Let $f\in L^\infty(\B)$, $g\in L^\infty(\partial \B)$. Then 
\[ \Lambda_\B (f,g)\geq \Lambda_\B\left(f^{\sharp},g^{\sharp}\right).\]
\end{proposition}
As we explained earlier, our main contribution consists in proving a Talenti inequality for the underlying cooperative system.  Let us first define the comparison relation relevant to us. 

\begin{definition}[Comparison relationship on spheres]\label{def:comparison_relation}    
Let $u_1, u_2\in L^\infty(\B)$ with $ u_1, u_2\geq 0$. We write
   \[ u_1\preceq u_2\] whenever one of the three following equivalent conditions holds:
   \begin{itemize}
       \item  For a.e. $r \in (0,1]$ and any $m_0\in (0,\mathcal H^{d-1}({\mathbb S^{d-1}_r}))$,
    \begin{equation*}
\max_{\substack{E\subset {\mathbb S^{d-1}_r},\\ \mathcal H^{d-1}(E)=m_0}}\int_E u_1\leq \max_{\substack{E\subset {\mathbb S^{d-1}_r},\\ \mathcal H^{d-1}(E)=m_0}}\int_E u_2.
    \end{equation*}
    \item For a.e. $r\in(0,1]$ and any $\theta\in (0,\th_{\max}^r)$,
    \[
    \int_{K_r(\theta)}u_1^\sharp\leq \int_{K_r(\theta)}u_2^\sharp
    \]
    \item For a.e. $r\in (0,1]$ and any convex increasing function $\phi:\R\to\R$,
    \[ \int_{{\mathbb S^{d-1}_r}} \phi(u_1)\leq \int_{{\mathbb S^{d-1}_r}}\phi(u_2).\]
    \end{itemize}
\end{definition}
The equivalence between these properties is standard, see \cite{ALT_optimization_prescribed_rearrangements}.
We are now able to state our main result. \BBB
\begin{theorem}
    \label{theo:rigidity}
    Let $f\in L^\infty(\B)$  with $f\geq0$, and set $\beta>0$. \\ \ \\
 \textbf{(i) Rigid Talenti inequality for scalar equations}:  Let $w\in L^\infty(\S)\,, w\geq 0$ and let $(u,v)$ solve
    \begin{equation}\label{Eq:theo1}
        \begin{cases}
            -\Delta u = f & \text{ in  }\B,  
             \\ { \partial_\nu u + \beta u = w }& \text{ on }\S,
        \end{cases}
        \qquad \text{ and } \qquad
        \begin{cases}
            -\Delta v = f^{\sharp} & \text{ in  }\B,  
             \\ \parnu v + \beta v = w^{\sharp} & \text{ on }\S.
        \end{cases}
    \end{equation}
Then \[ u\preceq v.\]

Furthermore,
    if there exists $p\in (1,\infty)$ such that 
   \[
        \norma{u}_{L^p(\B)} = \norma{v}_{L^p(\B)}
    \]
    then, up to the same rotation, $u=u^\sharp=v$, $f=f^{\sharp}$ and $w=w^\sharp$. \\
 \item \textbf{(ii) Rigid Talenti inequality for a fully coupled system}: Let $m_1\in L^\infty(\B)$ with $m_1\geq0$ and $\|m_1\|_{L^\infty(\B)}$ sufficiently small so that the lowest eigenvalue $\lambda_1^\beta(m_1)$ of $-\Delta -m_1$ endowed with the Robin boundary conditions $\partial_\nu\cdot+\beta\cdot=0$ is positive:
\begin{equation}\label{Eq:SmallPotential}
\lambda_1^\beta( m_1 )>0.
\end{equation}
Let $g, m_2\in L^\infty(\S)$ with $g, m_2\geq 0$.
Let  $(u,w_g)$ and $(v,w_{g^{\sharp}})$ solve the following Poisson problems
    \begin{equation}\label{Eq:theo2}
        \begin{cases}
            -\Delta u -m_1 u= f & \text{ in  }\B,  
             \\ \parnu u + \beta u = w_g  & \text{ on }\S,
            \\-\Delta_{ \Sigma} w_g+m_2w_g=g&\text{ on }\S,
        \end{cases}
        \qquad \text{ and } 
        \begin{cases}
            -\Delta v-m_1^{\sharp} v = f^{\sharp} & \text{ in  }\B,  
             \\ \parnu v + \beta v = w_{g^{\sharp}}  & \text{ on }\S,
            \\ -\Delta_{ \Sigma}  w_{g^{\sharp}}+(m_2)_{\sharp} w_{g^{\sharp}}=g^{\sharp}&\text{ on }\S.
        \end{cases}
    \end{equation} Then
    \[ u\preceq v.\]
    Furthermore,
    if there exists $p\in (1,\infty)$ such that 
   \begin{equation}\label{eq:equality_norms_rigidity}
        \norma{u}_{L^p(\B)} = \norma{v}_{L^p(\B)}
    \end{equation}
    then, up to the same rotation, $m_1=m_1^{\sharp}\,, m_2=(m_2)_{\sharp}\,, u=u^{\sharp}=v\,, g=g^{\sharp}\,, f=f^{\sharp}$.
\end{theorem}

\begin{remark}[About the $p=1$ case.]
On the other hand there can be no rigidity for $p=1$ (at least for scalar equations). Indeed, taking $w=0$ and any $f$, we keep on calling $u$ the solution of \eqref{Eq:theo1} and we denote by $q_\B$ the Robin torsion function of the ball, \emph{i.e.} $q_\B$ solves 
\[\begin{cases}
    -\Delta q_\B=1 & \text{in }\B,\\
    \partial_\nu q_\B+\beta q_\B=0 & \text{on } \S.\end{cases}
    \] 
    Then $\int_\B q_\B f=\int_\B q_\B f^\sharp$ as $q_\B$ is radial, so that $\int_\B u=\int_\B q_\B f=\int_\B q_\B f^\sharp=\int_\B v$. 
\end{remark}
Regarding the comparison statements, the second item of the theorem will be a consequence of the first one, which is itself based on the following result of J. J. Langford (see \cite[Theorem 1.2]{Lanford_Robin}).

\begin{theoremnonumbering}[Talenti comparison for Robin problems, \cite{Lanford_Robin}]
    \label{Th:Langford}
Let $\beta > 0$ and  $f \in L^\infty(\B)$ with $f\geq 0$. If $u\,, v$ solve    \begin{equation}
        \label{eq:syst_u,v_Rob}
        \begin{cases}
            -\Delta u = f & \text{ in  }\B,  
             \\ \parnu {u}+ \beta u = 0  & \text{ on }\S,
        \end{cases}
        \qquad \text{ and } \qquad
        \begin{cases}
            -\Delta v = f^{\sharp} & \text{ in  }\B,  
             \\ \parnu v + \beta v = 0  & \text{ on }\S.       \end{cases}
        \end{equation}
   Then
   \[  u\preceq v.\]
\end{theoremnonumbering}

 Roughly speaking, our approach for proving the concentration inequalities from Theorem \ref{theo:rigidity} consists in approximating the problem \eqref{Eq:theo1} with solutions of \eqref{eq:syst_u,v_Rob}-type problems. 

On the other hand, as we mentioned in the introduction, Langford left the question of rigidity in Theorem \ref{Th:Langford} open. In particular, it is not clear whether one can use the usual (when it comes to Talenti inequalities) strategy, which relies on the rigidity of the isoperimetric inequality. We rather prove rigidity using a different approach, which is more energetic in nature. The following rigidity result is the prototypical one (in particular, the rigidity result of Theorem \ref{theo:rigidity} is based on it).
\begin{theorem}[Rigidity of the Talenti inequality for Robin problems] \label{theo:rigidity1}
Retaining the notations of Theorem \ref{Th:Langford}, if there exists $p\in(1,\infty)$ such that 
\[ \Vert u\Vert_{L^p(\B)}=\Vert v\Vert_{L^p(\B)}\] then, up to the same rotation, $u=u^{\sharp}=v$ and  $f=f^{\sharp}$.
\end{theorem}

 The idea is to show that if equality occurs, then $f$ is a maximiser of an auxiliary optimisation problem where one can use the rigidity of the Poly\'a-Szeg\"{o} inequality.  

We finish this paragraph by underlining an interesting corollary of Theorem \ref{theo:rigidity}, which is the following comparison principle for a solution of the Dirichlet problem.
\begin{corollary}\label{Co:DirichletData}
Let $w\in W^{\frac12,2}(\S)\cap L^\infty(\S)$ with $w \geq 0$, and $f\in L^\infty(\B)$ with $f\geq 0$. Assume that $u$ and $v$ solve the following Poisson problems
    \[
        \begin{cases}
            -\Delta u = f & \text{ in  }\B,  
             \\ { u = w } & \text{ on }\S,
        \end{cases}
         \text{ and } 
        \begin{cases}
            -\Delta v = f^\sharp & \text{ in  }\B,  
             \\ { v = w^{\sharp} } & \text{ on }\S.
        \end{cases}
    \]
Then 
    \[
    u\preceq v.
    \]
\end{corollary}

\subsubsection{The shape optimisation problem}
In the second part of the paper, we study the spectral shape optimisation problem: $f$ and $g$ are fixed and we want to minimise $\Lambda_\O(f,g)$ with respect to $\O$. We focus on the case where $f$ and $g$ are constant potentials, say  $f\equiv -c_i\,, g \equiv c_b$  for some constants $c_i,c_b\in\R$, and investigate the minimisation problem
\begin{equation}\label{Eq:Pv2}
\inf\left\{\lambda_{c_i,c_b}(\Om),\O\subset \R^d \text{  bounded Lipschitz  open set with } |\O|=|\B|\right\}
\end{equation}
where $\lambda_{c_i,c_b}(\Om):=\Lambda_\O(-c_i,c_b)$ was defined in \eqref{eq:egvPb}. 
As such, the constants $--c_i$ and $c_b$   represent the intrinsic growth rate of the species $u$ and $v$  in $\O$ and on its boundary respectively. As a consequence, the bigger $c_i\in\R$ (respectively $c_b\in\R$) is, the more unfavorable (respectively favorable) the interior (respectively boundary) of the domain is.  

Observe that this problem is \emph{a priori} delicate: there is no natural class of domains $\O$ where one might expect existence of an optimal shape. But more importantly, the eigenvalue $\lambda(\Om)$ blends together two different types of PDEs  featuring distinct spectral behaviours:  a bulk PDE set in $\Om$  with Robin boundary conditions,  for which the ball is often a minimiser of the related spectral quantity, and a surface PDE for which the ball  tends to be a maximiser instead (see Section \ref{Se:Bib} and the references therein). Thus, rather than aiming for a complete description of the minimisers we will provide partial answers to the problem, by exhibiting two different regimes for the parameters $c_i,c_b$, thus leading to distinct types of assertions: we establish some non-existence results, while investigating on the other hand the local minimality of the ball through the study of first and second-order optimality conditions.  The two next theorems gather these results. 

The first result guarantees that in general, one cannot expect existence of a minimiser for the optimisation of $\lambda_{c_i,c_b}$ under a volume constraint. Precisely, it claims that there is no global existence for this minimisation problem
whenever the resources input on the boundary dominates the one in the interior, among either smooth or convex bounded open sets of fixed volume.
\begin{theorem}\label{Th:NonExistence}
    Let $c_i, c_b\in \R$ with $-c_b< c_i$. We denote by $\mathcal{O}_{\text{s}}$ (respectively $\mathcal{O}_{\text{c}}$) the class of bounded smooth  (respectively convex) open set with volume $|\B|$. Then 
\[
\inf_{\Om\in\mathcal{O}_{\text{s}}}\lambda_{c_i,c_b}(\Om)= \inf_{\Om\in\mathcal{O}_{\text{c}}}\lambda_{c_i,c_b}(\Om)=-c_b
\]
and these two minimisation problems do not admit a solution. 
\end{theorem}

In the second result we investigate the local optimality of the ball.
As we are optimising under volume constraint, the optimality conditions will involve some Lagrange multiplier $\mu\in \R$ associated to this constraint. Thus the relevant functional is rather the Lagrangian $\Lm_{c_i,c_b,\mu}:=\lambda_{c_i,c_b}-\mu\mathrm{Vol}$, where $\mathrm{Vol}$ denotes the volume functional. Before stating the result, let us introduce some notation : 
if $\Om\mapsto J(\Om)$ is a shape functional we will denote (provided it exists) by $J'(\Om)\cdot(\theta)$ and $J''(\Om)\cdot(\theta,\theta)$ the first and second Fréchet-derivatives at $0$ in the direction $\th$ of the functional 
\[
\theta\in\C^{2,\eta}(\R^d,\R^d)\mapsto J\left((\text{Id}+\th)(\Om)\right),
\]
where $\eta\in(0,1)$. 
\begin{theorem}[The ball does/does not verify second-order optimality conditions] \label{thm:local_ball}
Let $c_i,c_b\in\R$ and $\eta\in(0,1)$. Then the ball verifies the first-order optimality condition: there exists $\mu\in\R$ such that for all $ \th\in\C^{2,\eta}(\R^d,\R^d),$
\[\Lm_{c_i,c_b,\mu}'(\B)\cdot(\th)=0.
\]
\begin{enumerate}
    \item\label{it:notlocalmin} Assume that $c_i>-c_b$.     Then the ball does not verify the second-order optimality condition (under volume constraint): there exists $
\th\in \C^{2,\eta}(\R^d,\R^d)$ with $\int_\S \th\cdot\nu_\B=0$ such that
\[\Lm_{c_i,c_b,\mu}''(\B)\cdot(\th,\th)<0.
    \]
    \item \label{it:localmin}Assume $2\leq d\leq 5$, and $c_i\ll -c_b$, in the sense that $c_i+c_b\to-\infty$ . Then 
        the ball verifies the second-order optimality condition (under volume constraint): there exists $c>0$ (which can be chosen uniform in $(c_i\,, c_b)$ for $c_i+c_b$ negative enough) such that for all $\th\in\C^{2,\eta}(\R^d,\R^d)$ with $\int_\S \th\cdot\nu_\B=0$,
    \[
    \Lm_{c_i,c_b,\mu}''(\B)\cdot(\th,\th)\geq c\|\th\cdot\nu_\B\|_{{W^{1,2}}(\S)}^2.
    \]
\end{enumerate}
\end{theorem}

 This result relies on a thorough second-order study of the functional $\lambda_{c_i,c_b}$, thus involving the computations of its first and second shape derivatives and their analysis at the ball (see the introduction of Section \ref{sect:second_order} for more details).  The interest of this theorem is two-fold:  
\begin{itemize}
    \item 
in the case $c_i>-c_b$, \textit{i.e.} when the boundary of the domain is supposed to be more favorable than the interior, a regime for which there is no global minimiser of $\lambda_{c_i,c_b}$ (see Theorem \ref{Th:NonExistence}), then the ball does not satisfy second-order optimality conditions, 
\item  in dimensions $2\leq d\leq5$, in the regime $c_i\ll -c_b$, \textit{i.e.} when the interior of the domain is assumed to be much more favorable than the boundary, the ball verifies second-order optimality conditions. 
\end{itemize}

The fact that the second statement is restricted to low dimensions is mainly due to technical difficulties, as we expect that an analogous claim in higher dimensions should follow from an estimate of the first Robin eigenvalue of the ball.

Although we do not show a complete local minimality statement in the regime $c_i\ll-c_b$, we believe that here also the gap to fill is merely technical (see Remark \ref{rk:locmin}), and that in any case this result constitutes strong evidence of this claim. 

\subsection{Bibliographical references}\label{Se:Bib} 
We survey some related works.

\textbf{Optimisation problems in population dynamics.}
The study of optimal control and shape optimisation problems arising from population dynamics has been the subject of a great interest in recent years. It amounts to answering the following question: what is the best way to design an environment in order to optimise various quantities, be it the survival ability, the total population size, etc.  These optimisation problems can usually be sorted out into two categories. The first one is the class of non-energetic optimisation problems, such as the total population size, which directly deals with the non-linear population dynamics' model  (see \cite{Mazari_Nadin_Privat_KPP} and the references therein). The second class of problems, much closer to the one we study in this paper, is the class of energetic or spectral problems, where a prototypical example is given by the minimisation of the first eigenvalue of $-\Delta -V$ (with some constraints on $V$). While several qualitative results relying on the isoperimetric inequality and Talenti inequalities have been shown (see the next paragraph), this is an old problem which is still not yet fully understood. We refer to \cite{zbMATH02194918,CoxLipton,Cox1990,zbMATH05530397,LamboleyLaurainNadinPrivat} and the references therein.

\textbf{Talenti inequalities for cooperative systems.}
Talenti inequalities are $L^p$-comparison principles for elliptic or parabolic systems. A Talenti inequality writes, in its most general form, as follows: letting $L$ be an elliptic or parabolic differential operator whose coefficients might depend on $x$, and $f$ be a source term, consider the solution $u$ of the PDE
\[ Lu=f\text{ in }\O,\text{ along with some boundary conditions.}\] Then, letting $L^*$, $\Om^*$  and $f^*$  be ``symmetrised" versions of $L$, $\Om$ and $f$ (depending also on the boundary conditions considered), and $v$ be the solution of 
\[ L^* v=f^*\text{ in }\O^*,\text{ with compatible boundary condition},\] it holds
\[ \text{ For a certain range of exponents $p$, } \Vert u\Vert_{L^p(\O)}\leq \Vert v\Vert_{L^p(\O^*)}.\] The most celebrated Talenti inequality is the original one proven by G. Talenti (see \cite[Theorem 1]{zbMATH03531830}), which deals with the Laplacian with  homogeneous Dirichlet boundary conditions. More specifically, letting $\O=\B$ be the centered ball of radius 1 and letting, for any $f\in L^2(\B)$ with $f\geq 0$ the function $f^*$ be the Schwarz rearrangement of $f$ (essentially, $f^*$ is the only radially symmetric, non-increasing function that has the same distribution function as $f$), consider the two equations 
\[ \begin{cases}-\Delta u=f&\text{ in }\B, 
\\ u\in W^{1,2}_0(\O)\end{cases}\text{ and }\begin{cases}-\Delta v=f^*&\text{ in }\B,\\ v\in W^{1,2}_0(\B).\end{cases}\] Then it was established in \cite{zbMATH03531830} that 
\[ \forall p\in [1,\infty]\,, \Vert u\Vert_{L^p(\O)}\leq \Vert v\Vert_{L^p(\O)}.\]
These inequalities are powerful tools when handling spectral optimisation problems, including non-symmetric ones (see \cite{NadinSiam}) and have known, over the last four decades, several generalisations to other types of operators and of boundary conditions (see \cite{zbMATH07748335}). We underline the fact that boundary conditions are extremely important, and let us point out that recently, following the independent works of Langford in \cite{Lanford_Robin} on the one hand and of Alvino, Nitsch \& Trombetti in \cite{zbMATH07748335} on the other hand, a lot of effort has been devoted to understanding symmetrisation problems with Robin boundary conditions (see \cite{zbMATH07748335,zbMATH07727514,Lanford_Robin,Langford_Neumann,zbMATH07527824,zbMATH07916632}). The situation is much more delicate than with Dirichlet boundary conditions. As announced, our approach to Theorems \ref{theo:rigidity}--\ref{theo:rigidity1}  builds and extends the methodology and results of J. J. Langford in \cite{Lanford_Robin}. Note that as a side result, we recover as a consequence of our analysis a Talenti inequality which deals with a rearrangement of the solution of a problem with Dirichlet boundary conditions, thus giving interesting properties of the solution of 
\[\begin{cases}-\Delta u=f&\text{ in }\B\,, 
\\ u=w&\text{ on }\S\end{cases},\] where the terms $f$ and $w$ are spherically rearranged (see Corollary \ref{Co:DirichletData}). We also mention the recent article of Mondino \& Vedovato \cite{zbMATH07384538}, which provides a wide class of Talenti inequalities on manifolds. 

Let us say that this type of result is not specific to scalar equations. In fact, as became apparent in recent contributions (such as \cite{CNT,Girardin_Mazari}), when it comes to symmetry properties of elliptic systems the relevant structure is rather the fact that the associated operator enjoys a maximum principle. In general, one should therefore expect Talenti-type inequalities for cooperative systems. This idea was for instance used in \cite{Girardin_Mazari} for the optimal control of a system of cooperating species in a torus. In the case under consideration in this paper, the fact that the two species live in different domains (\textit{i.e.} in $\Om$ or its boundary $\Sigma$) complexifies the matter, and we have to be careful when handling it. 

\textbf{Shape optimisation problems involving surface terms.}
The final class of questions this article fits into is that of shape optimisation problems involving surfaces. When dealing with these, 
in contrast with the classical context of spectral shape optimisation 
there is an additional difficulty stemming from the fact that the elliptic system \eqref{eq:egvPb} we are concerned with involves an equation on a moving surface.  Indeed, when minimising $\Lambda_\O(f,g)$ with respect to $\O$ under a volume constraint (assuming for simplicity that $f$ and $g$ are fixed constants), one cannot expect that any of the usual techniques to establish existence work and, in general, existence does fail. More specifically, a major obstruction comes from the fact that the convergence of a sequence of sets $\{\O_k\}_{k\in \N}$ for any ``natural" topology (\textit{e.g.} Hausdorff, or the topology induced by the convergence of the resolvants) does not imply a strong enough convergence of the boundaries  $\{\Sigma_k\}_{k\in \N}$ that would allow to pass to the limit for the surface energy in the variational formulation \eqref{Eq:Eigenvalue},  unless some stringent, $\mathscr C^{1,1}$ bounds of the boundary are imposed. Such difficulties have already been encountered in the literature, and we refer for instance to \cite{zbMATH06163678,Dalphin_18} and also to \cite{arXiv:2206.04357} in which the authors consider energies depending on a PDE set on a variable surface. 

Let us now give some instances of spectral optimisation problems set on surfaces. One of the most famous examples is the Brock-Weinstock inequality proven in \cite{zbMATH01584638}, which states that among all Lipschitz  sets of fixed volume, the Euclidean ball maximises the first Steklov eigenvalue. Another paradigmatic example is given by the maximisation of the first non-trivial eigenvalue of the Laplace-Beltrami operator on a  two-dimensional surface with fixed area. Requiring that admissible surfaces are diffeomorphic to the $2$-sphere, the Hersch inequality shown in \cite{zbMATH03356292} states that this eigenvalue is maximised at the $2$-sphere.  Closest to our interest is the Wentzell-Laplace problem, which couples a diffusion inside the domain with one on its boundary. Namely, one aims at maximising in $\O$ the first eigenvalue $\sigma=\sigma(\O)$ associated with the problem
\[ \begin{cases}
-\Delta u=0&\text{ in }\O\,, 
\\ -\Delta_\Sigma u+\partial_\nu u=\sigma u&\text{ on } \Sigma,
\end{cases}\]
where $\Delta_\Sigma$ is the tangential Laplace operator on $\Sigma:=\partial \O$. Dambrine, Kateb \& Lamboley conjectured in \cite{DKL16} that the ball should maximise $\sigma$ under volume constraint. We emphasise the fact that their conjecture is backed up by a fine second-order analysis, from which we draw inspiration for proving Theorem \ref{thm:local_ball}.

\textbf{Structure of the paper.}
The paper is organised linearly. In Section \ref{Se:Talenti}, we prove the Talenti-type inequalities associated to the optimal control problem, namely Proposition \ref{Pr:FK} and Theorems \ref{theo:rigidity} and \ref{theo:rigidity1}. Section \ref{sect:shapediff} is devoted to showing the results in connection with the spectral shape optimisation problem, Theorems \ref{Th:NonExistence} and \ref{thm:local_ball}. 

\section{Optimal control and Talenti comparison principles}\label{Se:Talenti}
This section is dedicated to the proof of our Talenti comparison results, namely Theorems \ref{theo:rigidity} and \ref{theo:rigidity1}. 
\subsection{Preliminary material}
We start by recalling some standard inequalities and comparison principles regarding cap symmetrisation, for which we refer to \cite{zbMATH03953325,zbMATH05042914}. Although these statements might be less known than their analogues for the standard Schwarz symmetrisation in $\R^d$, let us note that they can be obtained in a straightforward manner using the observation that, in polar coordinates, the cap symmetrisation is nothing but a Steiner symmetrisation in all the angular coordinates (see also \cite{zbMATH07384538}).  Using integration in polar coordinates, the statements on $\S$ thus lead to similar equalities/inequalities on $\B$ that we also recall below. 

First, the equimeasurability of the rearrangement yields equality of the $L^p$ norms: if $p\in[1,\infty)$, for any $f\in L^p(\B)$ and $g\in L^p(\S)$ it holds
\begin{equation}\label{Eq:RearNorm}
\int_\B |f|^p=\int_\B |f^{\sharp}|^p\,, \int_{\S} |g|^p=\int_{\S}|g^{\sharp}|^p.\end{equation}  Furthermore, for any $f_1\,, f_2\in L^2(\B)$, $r\in(0,1)$ and $g_1\,, g_2\in L^2(\mathbb{S}^{d-1}_r)$, the Hardy-Littlewood inequality states that 
\begin{equation}\label{Eq:HL}
\int_\B f_1f_2\leq \int_\B f_1^{\sharp}f_2^{\sharp}, \text{ and for any }\th\in(0,\th_{\max}^r), \int_{K_r(\th)} g_1g_2\leq \int_{K_r(\th)}g_1^{\sharp}g_2^{\sharp}.
\end{equation}  This inequality is rigid in the following sense: if $g_1$ is not constant and if, for any $t\in\R$, $|\{g_2=t\}|=0$, then 
\begin{multline}\label{Eq:HLR}
\int_{\S} g_1g_2=\int_\S g_1^\sharp g_2^\sharp\\\text{  if, and only if, for any $t, \tau\in\R$, either $\{g_1>t\}\subset \{g_2>\tau\}$ or $ \{g_2>\tau\}\subset \{g_1>t\}$.}
\end{multline}
 In particular, if $g_2=g_2^\sharp$ is symmetric decreasing this implies that each $\{g_1>t\}$ is a spherical cap, giving in turn $g_1=g_1^\sharp$. 

 The Hardy-Littlewood inequality combined with equality of the norms implies that for any $f_1\,, f_2\in L^2(\B)$ and $g_1\,, g_2\in L^2(\S)$
\begin{equation}\label{Eq:Contraction}
 \int_\B (f_1^{\sharp}-f_2^{\sharp})^2\leq\int_\B (f_1-f_2)^2, \int_{\S}(g_1^{\sharp}-g_2^{\sharp})^2\leq \int_{\S}(g_1-g_2)^2.\end{equation}
 Finally, the Poly\'a-Szeg\"{o} inequality states that, letting $\nabsig$ denote the tangential gradient,
\begin{equation}\label{Eq:FK1}
\forall v \in W^{1,2}(\S)\,, \int_{\S}\left|\nabsig v^{\sharp}\right|^2\leq \int_{\S}|\nabsig v|^2\end{equation}
and 
\begin{equation}\label{Eq:FK2}
\forall u\in W^{1,2}(\B)\,, \int_\B \left|\n u^{\sharp}\right|^2\leq \int_\B|\n u|^2.\end{equation}
\begin{remark}
To be more specific, \eqref{Eq:FK2} follows from \eqref{Eq:FK1} and the results of Berestycki \& Lachand-Robert \cite[Lemma 2.10]{BerLach}, which guarantees that for any function $u \in W^{1,2}(\B)$, it holds 
\[ \int_\B |\n_ru^\sharp|^2\leq \int_\B |\n_ru|^2\,\]
where $\n_r$ denotes he radial part of the gradient.\end{remark}


\subsection{Proof of Proposition \ref{Pr:FK}}
Proposition \ref{Pr:FK} follows from a combination of the previous inequalities. 
\begin{proof}
Letting $(u,v)$ be an eigencouple associated with $\Lambda_\B(f,g)$ (\textit{i.e.} a solution of \eqref{eq:egvPb} for $\Om=\B$), it holds
\begin{align*}
\Lambda_\B(f,g)&=\frac{\int_\B |\n u|^2+\int_{\S} |\n v|^2-\int_\B fu^2-\int_{\S} gv^2+\int_{\S} (u-v)^2}{\int_\B u^2+\int_{\S} v^2}
\\&\geq\frac{ \int_\B |\n u^{\sharp}|^2+\int_{\S} |\nabsig v^{\sharp}|^2-\int_\B f^{\sharp}(u^{\sharp})^2-\int_{\S} g^{\sharp}(v^{\sharp})^2+\int_{\S} (u^{\sharp}-v^{\sharp})^2}{\int_\B (u^{\sharp})^2+\int_{\S} (v^{\sharp})^2}
\\&\text{ by \eqref{Eq:RearNorm}-\eqref{Eq:HL}-\eqref{Eq:Contraction}-\eqref{Eq:FK1}-\eqref{Eq:FK2}}
\\&\geq \Lambda_\B(f^{\sharp},g^{\sharp}),
\end{align*}
thus finishing the proof of the proposition.
\end{proof}
\subsection{Proof of Theorem \ref{theo:rigidity1}} 
This paragraph is dedicated to the proof of the rigidity of the Langford comparison result (recalled in Theorem \ref{Th:Langford}).
  In \cite{Lanford_Robin}, one of the main ideas for proving this inequality consists in showing that if $u:\B\to\R$ is a solution of a PDE with Robin boundary conditions, then its symmetrisation $u^\sharp$ is a subsolution of the symmetrised problem (see \cite[Theorem A.3]{Langford_Neumann}). As it will reveal important to prove Theorem \ref{theo:rigidity1}, in Lemma \ref{Le:SubHarmonicity} below we briefly recall a precised version of this result  (which follows Baernstein \cite{Baernstein_unified_approach} for the strategy of its proof). We also interpret it in terms of an optimal control problem  (in Remark \ref{rk:subh}), using the methodology of Burton in \cite{Burton}. It is important to observe that this result is what allows to tie rigidity in the Talenti inequality with optimality for an energetic optimal control problem  (see Remark \ref{rk:subh}).
  
 Throughout this section, for functions $\psi, \phi\in L^\infty(\B)$ we write $\psi\sim \phi$  if $\psi^\sharp=\phi^\sharp$.
 \begin{lemma}\label{Le:SubHarmonicity}
 Let $f\in L^\infty(\B)$ with $f\geq 0$, and let $u$ solve 
 \begin{equation}\label{Eq:LangfordInter1}
 \begin{cases}
 -\Delta u=f&\text{ in }\B\,, 
 \\ \partial_\nu u+\beta u=0&\text{ on }\S.\end{cases}\end{equation}
 Then $u^{\sharp}$ satisfies  $-\Delta u^\sharp\leq f^\sharp$ in the weak sense:
 \[ \forall\phi\in \mathscr C^2_c(\B)\,\text{ with } \phi\geq0 \text{ and }\phi^{\sharp}=\phi\,, -\int_\B u^{\sharp}\Delta\phi\leq \int_\B f^{\sharp} \phi.\]
 Furthermore, if 
 \[-\Delta u^{\sharp}=f^{\sharp}\text{ in }\B\] (meaning there is equality for any such $\phi$ in the above statement)  then for each $\phi\in \mathscr C^2_c(\B)$ with $\phi\geq0$ and $\phi^{\sharp}=\phi$, it holds 
\begin{equation}\label{Eq:L2} \forall\psi \in L^\infty(\B) \text{ with }  \psi\sim \phi\,,\left(\int_\B u\psi=\int_\B u^{\sharp}\phi\Rightarrow \int_\B f\psi=\int_\B f^{\sharp}\phi\right). \end{equation} 
 \end{lemma}
 
 \begin{proof}[Proof of Lemma \ref{Le:SubHarmonicity}]
The key observation, as in \cite{Langford_Neumann}, is the following: for any $\eta\in \mathscr C^2(\B)$ and any smooth, positive, radially symmetric decreasing kernel $K:\R^d\to\R$ with $\int_{\R^d}| x|^2K(x)dx=1$, letting $K_\e:=\e^{-d}K(\e^{-1} x)$ it holds 
\begin{equation}
    \label{eq: point_lap}
 \Delta \eta=\lim_{\e\to 0}\frac{(K_\e\star\eta)-\eta}\e,\end{equation} and the convergence is locally uniform in $\B$. This follows from a second-order Taylor expansion of $K_\eps\star \eta$ at a point $x\in \B$, noting that the first-order term vanishes since $K$ is radially symmetric.  Now let $u$ solve \eqref{Eq:LangfordInter1} and take $ \phi\in \mathscr C^2_c(\B)$ with $\phi\geq0$ and  $\phi^{\sharp}=\phi$. As $u$ and $u^{\sharp}$ are equimeasurable, the maps 
\[ T_1:\{ \psi\in L^\infty(\B) \,, \psi\sim \phi\}\ni \psi\mapsto \int_\B \psi u\text{ and }T_2:\{ \psi\in L^\infty(\B) \,, \psi\sim \phi\}\ni \psi\mapsto \int_\B \psi u^\sharp\] have the same range (as this range only depends on the distribution function of $u$, by the Cavalieri principle).  Thus, we can find some  $\psi\in L^\infty(\B)$  with $\psi\sim \phi$ such that 
\begin{equation}\label{eq:in} \int_\B \psi u=\int_\B \phi^{\sharp}u^{\sharp}.\end{equation} Now, let $\e>0$ be fixed.  We use the following Riesz-Sobolev-type inequality (see \cite[Corollary 4]{Baernstein_unified_approach}) 
\[
\int_{\B\times\B}u(x)\psi(y)K(x-y)\leq\int_{\B\times\B}u^\sharp(x)\psi^\sharp(y)K(x-y)
\]
and combine it with \eqref{eq:in} to get \BBB
\[  \int_\B (K_\e\star u-u)\psi=\int_\B u\left(K_\e\star\psi-\psi\right)\leq \int_\B u^{\sharp}\left(K_\e\star\phi^{\sharp}-\phi^{\sharp}\right).\]
Dividing the above by $\eps>0$, we take the limit as $\e\to 0$ and use \eqref{eq: point_lap} on each side of the inequality to get 
\[-\int_\B f \psi\leq \int_\B u^{\sharp}\Delta \phi^{\sharp}\]
 (note that although \eqref{eq: point_lap} is stated for $\C^2$ functions, it holds for $W^{2,2}$ functions by approximation). 
Applying
the Hardy-Littlewood inequality \eqref{Eq:HL} we obtain 
\[-\int_\B f^{\sharp}\phi^\sharp\leq-\int_\B f \psi\leq \int_\B u^{\sharp}\Delta \phi^{\sharp},\]
which constitutes the announced inequality. In the equality case, \textit{i.e.} if $-\Delta u^{\sharp}=f^{\sharp}$ in $\B$, then $-\int_\B f^{\sharp}\phi^\sharp= \int_\B u^{\sharp}\Delta \phi^{\sharp}$ and we therefore directly obtain from the above that 
\[\int_\B f\psi=\int_\B f^{\sharp}\phi^{\sharp},\]
thus concluding the proof of the lemma.
\end{proof} 

\begin{remark}[A comment on Lemma \ref{Le:SubHarmonicity}]\label{rk:subh}
The main interest of Lemma \ref{Le:SubHarmonicity} is that it allows to establish very quickly that if $-\Delta u^{\sharp}=f^\sharp$ then $f$ is the solution of an energy minimisation problem, which allows to derive the rigidity of the Talenti inequality from the rigidity of the P\'oly\`a-Szeg\"{o} inequality.  Indeed, condition \eqref{Eq:L2} implies that $f$ solves $\max_{\tilde f\sim f}\int_\B \tilde f u$. Burton \cite{Burton} showed that this implies the existence of an increasing function $F:\R\to \R$ such that $f=F(u)$. This is typically the optimality condition that one would expect from the minimisation of the energy associated with $u$, see \eqref{Eq:RobinEnergy} below.
\end{remark}

We are now ready to pass to the proof of Theorem \ref{theo:rigidity1}, which we divide in several steps. Recall that $f\in L^\infty(\B), w\in L^\infty(\S)$ with $f, w\geq 0$ and that $(u,v)$ is a solution to  \eqref{eq:syst_u,v_Rob} with the hypothesis that $\Vert u\Vert_{L^p(\B)}=\Vert v\Vert_{L^p(\B)}$. 

\begin{proof}[Proof of Theorem \ref{theo:rigidity1}]
We introduce the energy functional
 \begin{equation}\label{Eq:RobinEnergy} \mathcal E:L^\infty(\B)\times W^{1,2}(\B)\ni (\tilde f,\tilde u)\mapsto \frac12\int_\B |\n \tilde u|^2+\frac\beta2\int_{\S} \tilde u^2-\int_\B \tilde u\tilde f.
 \end{equation}
As $v$ is the unique minimiser of $\mathcal E(f^{\sharp},\cdot)$ among $W^{1,2}(\B)$ functions, then applying \eqref{Eq:RearNorm}, \eqref{Eq:HL} and \eqref{Eq:FK2} to $\mathcal E(f^{\sharp},\cdot)$ we have in particular that $v=v^\sharp$ in $\B$.

 \textbf{Step 1: $u^\sharp=v$.} Let us start by proving that 
 \begin{equation}\label{Eq:EqualRearrangements}
 u^{\sharp}=v \text{ on }\B.
 \end{equation}
Introduce 
 \[\mathcal C:=\{ \psi\in L^\infty(\B)\,,\psi\geq 0\,, \psi\preceq v\}.\]  The set $\mathcal C$ is convex, and its extreme points are given by the set 
 \[ \mathcal K:=\{\psi\in L^\infty(\B)\,, \psi^{\sharp}=v\}\] (for these results, see \cite[Section 2]{ALT_optimization_prescribed_rearrangements}). The map $T:\mathcal C\ni \psi\mapsto \Vert \psi\Vert_{L^p(\B)}^p$ is strictly convex, and its maximisers are consequently all extreme points of $\mathcal C$. 
 Theorem \ref{Th:Langford} guarantees that $u\in \mathcal C$, while the assumption that $\Vert u\Vert_{L^p(\B)}=\Vert v\Vert_{L^p(\B)}$ implies that $u$ maximises $T$, since we know that $\max_{\mathcal C}T=T(v)$ by definition of $\mathcal{C}$. The equality \eqref{Eq:EqualRearrangements} follows.

\textbf{Step 2: $u=u^{\sharp}=v\text{ on }\S$.} As $u^\sharp=v$ in $\B$ by \eqref{Eq:EqualRearrangements}, in particular
 \[ \int_\B u^2=\int_\B v^2.\] Since in addition $v=v^\sharp$ we may apply Lemma \ref{Le:SubHarmonicity} with  $\phi=v\,, \psi=u$ to obtain 
 \begin{equation}\label{Eq:EnergyEqual}
\int_\B uf=\int_\B v f^{\sharp}. 
 \end{equation}
     Note that although Lemma \ref{Le:SubHarmonicity} is only stated for compactly supported functions $\phi$, \eqref{Eq:EnergyEqual} is derived using a straightforward approximation argument (multiplying by a compactly supported, radially symmetric cut-off function).

As they solve \eqref{eq:syst_u,v_Rob}, $u$ and $v$ minimise respectively $\mathcal E(f,\cdot)$ and $\mathcal E(f^{\sharp},\cdot)$ among $W^{1,2}(\B)$ functions. Using the weak formulation of \eqref{eq:syst_u,v_Rob}, we deduce 
 \[ \min_{\tilde u\in W^{1,2}(\B)}\mathcal E(f,\tilde u)=\mathcal E(f,u)=-\frac12\int_\B fu\,,\min_{\tilde v\in W^{1,2}(\B)}\mathcal E(f^{\sharp},\tilde v)=\mathcal E(f^{\sharp},v)=-\frac12\int_\B f^{\sharp}v.\]
The Poly\'a-Szeg\"{o} \eqref{Eq:FK2} and the Hardy-Littlewood \eqref{Eq:HL} inequalities give 
 \[-\frac12\int_\B uf= \mathcal E(f,u)\geq \frac12\int_\B |\n u^{\sharp}|^2+\frac\beta2\int_\S (u^{\sharp})^2-\int_\B f^{\sharp} u^{\sharp}=\mathcal E(f^{\sharp},v)=-\frac12\int_\B vf^{\sharp}.\] From \eqref{Eq:EnergyEqual} we deduce that these two inequalities are in fact equalities, so that in particular there is equality in the Poly\'a-Szeg\"{o} inequality  for a.e. sphere, \textit{i.e.} $\int_{\mathbb S^{d-1}_r}|\nabla_{\Sigma_r} u|^2= \int_{\mathbb S^{d-1}_r}|\nabla_{\Sigma_r} u^\sharp|^2$ for a.e. $r\in(0,1)$. Writing in polar coordinates, the rigidity of the latter \cite[Theorem 4.3]{zbMATH07384538} implies that, for a.e. $r\in (0,1)$, there exists an angle $\theta_r$ such that 
 \[ \forall \th\in(0,\th_{\max}^r), \ u(r,\theta)=u^{\sharp}(r,\theta+\theta_r)=v(r,\theta+\theta_r).\] 
In particular, passing to the limit as $r\to1$, up to a rotation we thus have 
\[ u=u^{\sharp}=v\text{ on }\S.\]

 \textbf{Step 3: $u=u^{\sharp}=v$ and $f=f^\sharp \text{ in }\B$.} We now prove that this implies (up to a rotation)
\begin{equation}\label{Eq:NDC}
u=u^{\sharp}=v\text{ in }\B.\end{equation}
Pick any $s\in (0,\mathcal H^{d-1}(\S))$ and define the spherical cap $E_s=E_s^{\sharp}\subset\S$ chosen so that $\mathcal H^{d-1}(E_s)=s$. Let $q_s$ be defined as 
\begin{equation}\label{Eq:Adjoints}
\begin{cases}
-\Delta q_s=0&\text{ in }\B\,, 
\\ \partial_\nu q_s+\beta q_s=\mathds 1_{E_s}&\text{ on }\S.
\end{cases}
\end{equation}We have
\[\int_\B fq_s=\int_\S \mathds 1_{E_s} u=\int_\S \mathds 1_{E_s} v=\int_\B f^{\sharp} q_s.
\]
Using an energy argument as the one used above for proving $v=v^\sharp$, we show again that for any $r\in (0,1]$, $q_s(r,\cdot)$ is symmetric decreasing in its angular coordinate. In combination with \eqref{Eq:HLR}, this implies that up to a rotation $f=f^\sharp$, thereby concluding the proof of rigidity.
\end{proof}

\subsection{Proof of Theorem \ref{theo:rigidity}}
There are two parts in Theorem \ref{theo:rigidity}, themselves divided in a comparison and a rigidity statement, which will all be proved independently.
\begin{proof}[Proof of the comparison result in item (i)]
We first extend the function $w\in L^\infty(\S)$ into a $L^\infty(\B)$ function (that we still call $w$) by setting     \[
        w(x) :=
        \begin{cases}
             w\left(\frac{x}{\norma{x}}\right) & \text{ if } x \in \B \setminus \{  0 \}, \\
            0 & \text{ if } x = 0.
        \end{cases}
    \]
Now, for any $\e\in(0,1)$, we define the annulus $\mathbb A_\e:=\{x \in \B\,, 1-\e\leq \Vert x\Vert\leq 1\}$, set 
\[ w_\e:=\frac1\e w\mathds 1_{\mathbb A_\e}\] and consider the solutions $u_\e\,, v_\e$ of 
    \begin{equation}
        \label{eq:definition_u_eps_and_v_eps}
        \begin{cases}
            -\Delta u_{\eps} =f+ w_\e& \text{ in } \B, \\
            \partial_\nu u_\e+\beta u_\e=0&\text{ on }\S        \end{cases}
        \qquad
        \begin{cases}
       -\Delta v_\e=f^{\sharp}+w_\e^{\sharp}&\text{ in }\B\,, 
       \\ \partial_\nu v_\e+\beta v_\e=0&\text{ on }\S.        \end{cases}
    \end{equation}
We are going to show that $\{u_\e\}_{\e\to 0}\,, \{ v_\e\}_{\e \to 0}$ converge, respectively, to $u$ and $v$ as $\e \to0$. Let us start by proving the uniform $W^{1,2}$ bounds
\begin{equation}\label{Eq:SobolevBound}
\limsup_{\e \to 0}\left(\Vert u_\e\Vert_{W^{1,2}(\B)}+\Vert v_\e\Vert_{W^{1,2}(\B)}\right)<\infty.
\end{equation}

\textbf{Step 1: proof of \eqref{Eq:SobolevBound}.}

We focus on showing that 
\[\limsup_{\e\to 0}\Vert u_\e\Vert_{W^{1,2}(\B)}<\infty,\] and the same proof will provide a similar bound for $v_\e$. First observe that there exists a constant $C>0$ such that for any $\delta\in(0,\frac{1}{2})$, for any $\p\in W^{1,2}(\B)$ it holds 
\begin{equation}\label{Eq:Trace}
\Vert \p\Vert_{L^2(\mathbb{S}^d_{1-\delta})}\leq C \Vert \p\Vert_{W^{1,2}(\B)}.\end{equation} Indeed, for $\delta\in (0,\frac{1}{2})$ the spheres $\mathbb{S}^d_{1-\delta}$ are uniformly Lipschitz and have uniformly bounded perimeters, which suffices (see \cite[Theorem 18.15]{zbMATH07647941}) to guarantee the uniformity of the above trace constant. Second, recall that the norms 
\begin{equation}\label{Eq:Norms}\Vert \p\Vert_{W^{1,2}(\B)}:=\int_\B |\n \p|^2+\int_{\B}\p^2\text{ and } \Vert \p\Vert_{N}^2:=\int_\B |\n \p|^2+\int_{\S}\p^2\end{equation} are equivalent (this is simply a consequence of the positivity of the first Robin eigenvalue). Thus, using $u_\e$ as a test function in the weak formulation of the equation on $u_\e$ we have
\begin{align*}
\int_\B|\n u_\e|^2+\beta\int_{\S}u_\e^2&=\int_\B fu_\e+\int_\B w_\e u_\e
\\&=\int_\B fu_\e+\frac1\e\int_{1-\e}^1 \left(\int_{\mathbb{S}^d_{r}} w\left(\frac{x}{\Vert x\Vert}\right) u_\e\right)dr 
\\&\leq\Vert f\Vert_{L^2(\B)}\cdot\Vert u_\e\Vert_{L^2(\B)}+\frac{C}\e \int_{1-\e}^1\left\Vert w\left(\frac{\cdot}{|\cdot|}\right)\right\Vert_{L^2(\mathbb{S}^{d-1}_{r})}\cdot\Vert u_\e\Vert_{W^{1,2}(\B)}dr
\\&\text{ thanks to \eqref{Eq:Trace},}
\\&\leq C'\left(\Vert f\Vert_{L^2(\B)}+\Vert w\Vert_{L^2(\S)}\right)\Vert u_\e\Vert_{W^{1,2}(\B)},
\end{align*} for some other constant $C'>0$. The equivalence of the norms defined in \eqref{Eq:Norms}  gives \eqref{Eq:SobolevBound}. 

\textbf{Step 2: convergence of the approximations.}

 From \eqref{Eq:SobolevBound} we deduce that there exist $u_0\,, v_0\in W^{1,2}(\B)$ such that, up to a subsequence, $u_\e\underset{\e\to 0}\rightarrow u_0\,, v_\e\underset{\e\to 0}\rightarrow v_0$ weakly in $W^{1,2}(\B)$, strongly in $L^2(\B)$. Up to further subsequence, it follows once again from \eqref{Eq:Trace} and integration in polar coordinates  that \[ \forall \p\in W^{1,2}(\B)\,, \int_\B w_\e \p\underset{\e\to 0}\rightarrow \int_{\S} w\p.\] Passing to the limit in the weak formulations of \eqref{eq:definition_u_eps_and_v_eps} thus proves that $u_0=u\,, v_0=v$ (where we recall that $(u,v)$ solve \eqref{Eq:theo1}). Finally, observe that from \eqref{Eq:Contraction}  one also has 
\begin{equation}\label{Eq:Contra2}
\lim_{\e\to0} u^\sharp_\e=u^\sharp \text{ and } \lim_{\e\to0} v^\sharp_\e=v^\sharp \text{ in }L^2(\B). 
\end{equation}
Let us now apply Theorem \ref{Th:Langford} to the approximations $u_\e,v_\e$. 
For a.e. $r\in(0,1)$, any $\th\in(0,\th_{\max}^r)$ and any $\e>0$,
\begin{align*}
\int_{K_r(\theta)} u_\e^{\sharp}
\leq \int_{K_r(\theta)} v_\e^{\sharp} \text{ by Theorem \ref{Th:Langford}}.
\end{align*}
 Using the co-area formula (and up to further subsequence), one deduces from \eqref{Eq:Contra2} that $u^\sharp_\e\to u^\sharp$ and $v^\sharp_\e\to v$ in $L^2(\mathbb{S}^{d-1}_r)$ for a.e. $r\in(0,1)$.  We can therefore pass to the limit in this inequality to get,  for a.e.  $r\in (0,1)$ and any $\th\in(0,\th_{\max}^r)$, 
\[ \int_{K_r(\theta)} u^{\sharp}\leq \int_{K_r(\theta)} v^{\sharp}.\] The conclusion follows.
    \end{proof}

\begin{remark}
     Note that we have used in the second step of the proof without further justification a (slightly) strengthened form of Theorem \ref{Th:Langford} (proven exactly in the same way): if $u,v$ solve    
    \begin{equation*}
        \begin{cases}
            -\Delta u = f & \text{ in  }\B,  
             \\ \parnu {u}+ \beta u = 0  & \text{ on }\S,
        \end{cases}
        \qquad \text{ and } \qquad
        \begin{cases}
            -\Delta v = \tilde f & \text{ in  }\B,  
             \\ \parnu v + \beta v = 0  & \text{ on }\S,       \end{cases}
        \end{equation*}
where $f$ and $\tilde f$ are non-negative functions in $L^\infty(\B)$ verifying $f\preceq \tilde f=(\tilde f)^\sharp$, then the concentration inequality $u\preceq v$ holds.
\end{remark}

 \begin{proof}[Proof of the rigidity result in item (i)]

 Assume that there exists $p\in(1,\infty)$ such that $\Vert u\Vert_{L^p(\B)}=\Vert v\Vert_{L^p(\B)}$.

 Since $u\preceq v$ thanks to the comparison result from item (i) which was proven above, we can show exactly in the same way as we did in \eqref{Eq:EqualRearrangements} that this implies $u^{\sharp}=v$ in $\B$.
As a consequence, using the same Cavalieri principle argument as in the proof of Lemma \ref{Le:SubHarmonicity}, we can find $(w'\,, f')\in L^\infty(\S)\times L^\infty(\B)$ such that $f'\sim f\,, w'\sim w$ and 
\begin{equation}\label{eq:dearrang_f_w}
    \int_\B  f'u+\int_\S  w'u=\int_\B f^\sharp v+\int_\S w^\sharp v.\end{equation} 
Let us introduce the energy functional
\begin{equation}\label{Eq:RobinEnergy2} \mathcal E:L^\infty(\S)\times L^\infty(\B) \times W^{1,2}(\B)\ni (\tilde w,\tilde f,\tilde u)\mapsto \frac12\int_\B |\n \tilde u|^2+\frac\beta2\int_{\S} \tilde u^2-\int_\S \tilde w\tilde u-\int_\B \tilde f\tilde u.
 \end{equation}
From \eqref{Eq:RearNorm}, \eqref{Eq:HL}, \eqref{Eq:FK2} and the weak formulation of \eqref{Eq:theo1}, it holds 
\begin{align*}-\frac12\int_\S w^\sharp v-\frac12\int_\B f^\sharp v&=\min_{\tilde w\sim w, \tilde f\sim f, \tilde u\in W^{1,2}(\B)\ } \mathcal E(\tilde w,\tilde f,\tilde u)\\
&\leq \mathcal E(w',f',u)=\frac12\int_\S wu+\frac12\int_\B fu-\int_\S w'u-\int_\B f'u.\end{align*} Thanks to \eqref{eq:dearrang_f_w} this gives 
\[ \int_\S w^\sharp v+\int_\B f^\sharp v\leq \int_\S w u+\int_\B fu,\] so that the equality 
\[\int_\S w^\sharp v+\int_\B f^\sharp v= \int_\S w u+\int_\B fu\] holds. We now proceed exactly as in Step 2 and 3 of the proof of Theorem \ref{theo:rigidity1} to conclude: first, this equality rewrites 
\[
\mathcal{E}(w^\sharp,f^\sharp,v)=\mathcal{E}(w,f,u),
\]
and the equality case in the Poly\'a-Szeg\"{o} inequality \eqref{Eq:FK2} thus implies that up to a rotation $u=u^{\sharp}=v\text{ on }\S$; then, using the function $q_s$ (introduced in \eqref{Eq:Adjoints})
\[
\int_\B fq_s+\int_\S wq_s= \int_\B f^\sharp q_s+\int_\S w^\sharp q_s, 
\]
and the equality case in the Hardy-Littlewood inequality \eqref{Eq:HLR} enables to conclude that up to the same rotation $f=f^\sharp$ and $w=w^\sharp$. 
 \end{proof}

 \begin{remark}\label{rk:loosen_rigid_compar}
     The very same arguments enable to show the same  comparison and rigidity results if $u,v$ verify instead
    \[
        \begin{cases}
            -\Delta u = f & \text{ in  }\B,  
             \\ { \partial_\nu u + \beta u = w }& \text{ on }\S,
        \end{cases}
        \qquad \text{ and } \qquad
        \begin{cases}
            -\Delta v = \tilde f & \text{ in  }\B,  
             \\ \parnu v + \beta v = \tilde w & \text{ on }\S,
        \end{cases}
    \]
for non-negative functions $f,\tilde f\in L^\infty(\B)$ and $w,\tilde w\in L^\infty(\S)$ with $f\preceq\tilde f=(\tilde f)^\sharp$ and $w\preceq\tilde w=(\tilde w)^\sharp$.
 \end{remark}
 
 We now iterate this first item to derive the comparison result for the fully coupled system \eqref{Eq:theo2}.
 
 \begin{proof}[Proof of item (ii)]
 We prove the comparison result and the rigidity one separately. For the former, we begin with the study of an auxiliary system.

 \textbf{Proof of the comparison result for an auxiliary system.}
 
  Let us first prove that if $(\tilde{u},\tilde{v})$ solve
 \begin{equation}\label{Eq:theo3}
 \begin{cases}
 -\Delta \tilde{u}-m_1 \tilde u=f&\text{ in }\B\,, 
 \\ \partial_\nu \tilde u+\beta \tilde u=w&\text{ on }\S\,, 
 \end{cases}\begin{cases} -\Delta \tilde v-m_1^\sharp\tilde v=f^{\sharp}&\text{ in }\B\,, 
 \\ \partial_\nu \tilde v+\beta \tilde v=w^{\sharp}&\text{ on }\S\end{cases}
 \end{equation}
 then 
 \begin{equation*}
 \tilde u\preceq \tilde v.
 \end{equation*}
 We follow an iterative procedure. Namely, consider $\{\tilde u_k\}_{k\in \N}\,, \{\tilde v_k\}_{k\in \N}$ defined iteratively as 
 \begin{multline}
 \tilde u_0=\tilde v_0\equiv 0\text{ and, for any $k\geq 1$, }\\\begin{cases}-\Delta \tilde u_{k+1}=f+m_1 \tilde u_k&\text{ in }\B\,, 
 \\ \partial_\nu \tilde u_{k+1}+\beta \tilde u_{k+1}= w &\text{ on }\S\,,\end{cases}\begin{cases}-\Delta \tilde v_{k+1}=f^{\sharp}+m_1^{\sharp} \tilde v_k&\text{ in }\B \,, \\ \partial_\nu \tilde v_{k+1}+\beta \tilde v_{k+1}= w^\sharp.\end{cases}\end{multline} Let us show that, for any $k\geq 1$, we have 
 \begin{equation}\label{Eq:Auxiliary2}
 \tilde u_k\leq \tilde u_{k+1}\leq \tilde u\,, \tilde v_k\leq \tilde v_{k+1}\leq \tilde v \text{ in }\B,
 \end{equation}
 and 
 \begin{equation}\label{Eq:Auxiliary3}
 \tilde u_k\preceq \tilde v_k.\end{equation}
We only prove \eqref{Eq:Auxiliary2} for the sequence $\{\tilde u_k\}_{k\in \N}$ as the proof is similar for $\{\tilde v_k\}_{k\in \N}$.
We proceed by induction. First, since $f\geq 0$  and $w\geq0$  the maximum principle implies $\tilde u_1\geq 0=\tilde u_0$, while the hypothesis $\lambda_1^\beta(m_1)>0$  ensures that the operator $-\Delta-m_1$ has a strong maximum principle, so that $\tilde u\geq 0$. Thus, setting $\bar u_1:=\tilde u_1-\tilde u$ we obtain 
\[-\Delta \bar u_1-m_1\bar u_1=-m_1\tilde u\leq 0 \text{ in }\B,\ \partial_\nu \bar u_1+\beta \bar u_1=0 \text{ on }\S,\] 
using also that $m_1\geq0$. The maximum principle implies $\bar u_1\leq 0$, which amounts to $\tilde u_1\leq \tilde u$. Now, assume that for $k\in \N\,, k\geq 1$ we have $\tilde u_{k-1}\leq \tilde u_k\leq \tilde u$. Set $\bar z_{k+1}:=\tilde u_{k+1}-\tilde u_k$. Since $\tilde u_{k-1}\leq \tilde u_k$, then $-\Delta \bar z_{k+1}\leq 0$ and $\partial_\nu \bar z_{k+1}+\beta\bar z_{k+1}=0$, thus deducing from maximum principle that $\bar z_{k+1}\leq 0$. Similarly, from $\tilde u_k\leq \tilde u$ we get that $\bar u_{k+1}:=\tilde u_{k+1}-\tilde u$ is non-positive, \textit{i.e.} that $\tilde u_{k+1}\leq \tilde u$. This concludes the proof of \eqref{Eq:Auxiliary2}.
On the other hand, \eqref{Eq:Auxiliary3} is proven by applying iteratively Theorem \ref{Th:Langford}, using the basic fact that for any two non-negative functions $f_1\,, f_2:\B\to\R$, $(f_1+f_2)^\sharp\preceq f_1^\sharp+f_2^\sharp$ and the Hardy-Littlewood inequality \eqref{Eq:HL}, as one also has $\tilde v=(\tilde v)^\sharp$ (by the same argument as in the proof of Theorem \ref{theo:rigidity1}).
We now show that 
\begin{equation}\label{Eq:Auxiliary4}
\tilde u_k\underset{k\to \infty}\rightarrow \tilde u\,, \tilde v_k\underset{k\to \infty}\rightarrow \tilde v\text{ weakly in $W^{1,2}(\B)$, strongly in $L^2(\B)$}.\end{equation}
To establish \eqref{Eq:Auxiliary4} observe first that 
\begin{equation*}
\limsup_{k\to \infty}\left(\Vert \tilde u_k\Vert_{W^{1,2}(\B)}+\Vert \tilde v_k\Vert_{W^{1,2}(\B)}\right)<\infty.
\end{equation*}
In fact, using first \eqref{Eq:Auxiliary2}  gives a uniform $L^\infty$ bound on $ -\Delta \tilde u_k$.  As
\[
\int_\B|\n \tilde u_k|^2+\beta\int_{\S}\tilde u_k^2=\int_\B (-\Delta \tilde u_k)\tilde u_k+\int_\S w \tilde u_k
\]
the equivalence of norm \eqref{Eq:Norms} gives the above bound on $\tilde u_k$ (the bound on $\tilde v_k$ is proven analogously). 
We deduce that up to subsequence $\{\tilde u_k\}_{k\in \N}$ converges weakly in $W^{1,2}(\B)$, strongly in $L^2(\B)$, to a function $\tilde u_\infty$. Passing to the limit in the weak formulation of the equation on $\tilde u_k$, we deduce that $\tilde u_\infty=\tilde u$  by uniqueness of the solution to \eqref{Eq:theo3} \BBB, yielding \eqref{Eq:Auxiliary4}.  Passing to the limit in \eqref{Eq:Auxiliary3} is then done as in the proof of the first item of Theorem \ref{theo:rigidity}, thus ensuring the announced comparison $ \tilde u\preceq \tilde v$.

 Let us comment on the fact that the very same proof again provides the comparison $ \tilde u\preceq \tilde v$ if $\tilde u,\tilde v$ verify this time
\[
 \begin{cases}
 -\Delta \tilde{u}-m_1 \tilde u=f&\text{ in }\B\,, 
 \\ \partial_\nu \tilde u+\beta \tilde u=w&\text{ on }\S\,, 
 \end{cases}\begin{cases} -\Delta \tilde v-m_1^\sharp\tilde v=f^{\sharp}&\text{ in }\B\,, 
 \\ \partial_\nu \tilde v+\beta \tilde v=\tilde w&\text{ on }\S,\end{cases}
 \]
 where $\tilde w\in L^\infty(\B)$ is non-negative with the hypothesis $w\preceq \tilde w=(\tilde w)^\sharp$. 

 \textbf{Proof of the comparison result for \eqref{Eq:theo2}.}

 Using a Talenti inequality on the sphere $\S$ (see \cite[Section 3]{zbMATH07384538}) it holds $w_g\preceq w_{g^\sharp}$. This in turn implies $u\preceq v$ thanks to the result we just proved (and the fact that $w_{g^\sharp}=(w_{g^\sharp})^\sharp$ again by an energetic argument). 
 
 \textbf{Proof of rigidity.} 
 
 We now assume that for some $p>1$ it holds $\Vert u\Vert_{L^p(\B)}=\Vert v\Vert_{L^p(\B)}$.  Using the ridigidy result from the previous item and Remark \ref{rk:loosen_rigid_compar} we get that up to the same rotation
 \[ u=u^\sharp=v\,, (f+m_1 u)=(f+m_1u)^{\sharp}=f^{\sharp}+m_1^{\sharp}u^\sharp\,, w_g=w_{g^{\sharp}}.\]
   In particular, this implies
  \[ f^\sharp-f=m_1u^\sharp-m_1^\sharp u=(m_1-m_1^\sharp)u^\sharp.\] Now, fixing $r\in (0,1)$ and $\theta\in (0,\th_{\max}^r)$, it holds
  \begin{align*}
   0&\leq \int_{K_r(\theta)}(f^\sharp-f),\text{ by definition}
   \\&= \int_{K_r(\theta)}(m_1-m_1^\sharp)u^\sharp&
   \\&\leq0,\text{ by the Hardy-Littlewood inequality \eqref{Eq:HL},}
  \end{align*} 
  so that we have equality $\int_{K_r(\theta)}f=\int_{K_r(\theta)}f^\sharp$. We claim that this gives that for any decreasing $\p:[0,\th_{\max}^r)\to\R$ (identifying $\p$ with the naturally associated function defined on $\mathbb S^{d-1}_r$ with constant value on each $\partial K_r(\th)$) we have 
\begin{equation}
    \label{eq:phi_rad_ffsharp}
\int_{\mathbb S^{d-1}_r} \p f=\int_{\mathbb S^{d-1}_r} \p f^\sharp.\end{equation}
Indeed, differentiating the previous equality in $\th$ we first deduce that for any $\theta\in (0,\th_{\max}^r)$
\[ \int_{\partial K_r(\theta)}f=\int_{\partial K_r(\theta)}f^\sharp\] and the co-area formula then yields 
\[ \int_{\mathbb S^{d-1}_r}\p f=\int_0^\infty \frac{t}{\p'(t)}\left(\int_{\partial K_r(\theta_t)}f\right)dt=\int_{\mathbb S^{d-1}_r}\p f^\sharp\] where for any $t\in\R$ the angle $\theta_t$ is chosen so that $\{\p=t\}=\partial K_r(\theta_t)$.
Picking any decreasing $\p$ and using the rigidity case of the Hardy-Littlewood inequality \eqref{Eq:HLR} on the just proven equality \eqref{eq:phi_rad_ffsharp}, we deduce $f=f^\sharp$. 
 This gives in turn  
\[ m_1^\sharp u^\sharp=m_1 u^\sharp.\]
As $u^\sharp>0$ by the strong maximum principle (unless $f=0$ and $w=0$, in which case there is nothing to prove) , we deduce
\[ m_1=m_1^\sharp.\] Noting that $w_{g^\sharp}=(w_{g^\sharp})^\sharp$ is symmetric (thanks to an analogous energetic argument as the one used for proving $v=v^\sharp$ at the beginning of the proof of Theorem \ref{theo:rigidity1}), the same reasoning leads to $ g=g^\sharp$ and $m_2=(m_2)_{\sharp}$, thus finishing the proof of the theorem.

 \end{proof}

\section{Spectral shape optimisation of the cooperative system}\label{sect:shapediff}
This section is dedicated to the study of the shape optimisation problem \eqref{Eq:Pv2}, meaning that we minimise \(
\lambda_{c_i,c_b}(\Om)\) while keeping $|\Om|$ fixed, $\O$
where $\Om$ is a Lipschitz bounded open set. Let us recall the variational formulation of $\lambda_{c_i,c_b}(\Om)$:
\begin{equation}
    \label{eq:lambda_variat}
    \lambda_{c_i,c_b}(\Om)=\min_{\substack{ (u,v)\in {W^{1,2}}(\Om)\times {W^{1,2}}(\Sigma)\\ (u,v)\ne(0,0)}}\mathcal{E}^\Om_{c_i,c_b}(u,v)\end{equation}
    where we denote
    \[\mathcal{E}^\Om_{c_i,c_b}(u,v):=
    \frac{ \displaystyle \int_\Omega|\nabla u|^2+c_i\int_\Omega u^2+\int_\Sigma |\nabla_\Sigma v|^2+\int_\Sigma (u-v)^2-c_b\int_\Sigma v^2}{\displaystyle \int_\Omega u^2+\int_\Sigma v^2}.
\]
 The Lipschitz character of $\Om$ is enough to ensure  existence and uniqueness (up to multiplication by a positive constant) of a minimiser $(u,v)$ to \eqref{eq:lambda_variat}, with $u>0$ in $\Om$ and $v>0$ in $\Sigma$. Throughout, the unique couple $(u,v)$ such that $\int_\Om u^2+\int_\Sigma v^2=1$ is called the associated eigencouple.

Alternatively, if $\Om$ is further assumed to be $\C^{2,\eta}$ for some $\eta\in(0,1)$, then $\lambda_{c_i,c_b}(\Omega)$ can be defined through Krein-Rutman Theorem as the principal eigenvalue of the system with unknowns $(u,v,\lambda)$:
\begin{equation}\label{eq:system_egv}
\begin{cases}
    -\Delta u=(\lambda-c_i)u & \text{in }\Omega,\\
    -\Delta_\Sigma v=(\lambda+c_b)v+u-v &\text{over }\Sigma,\\
    \partial_\nu u+u=v &\text{over }\Sigma,\\
    u>0 \text{ in }\Om,\ v>0 \text{ over }\Sigma,
\end{cases}
\end{equation}
(see \cite[Theorem 2.2]{BGT} and \cite[Proposition 3.3]{GMZ15}), in which case the so-called principal eigencouple $(u,v)$
belongs to $\C^{2,\eta}(\ov{\Om})\times\C^{2,\eta}(\Sigma)$. 
\subsection{Preliminary estimates}
We start by proving the following proposition, which provides useful bounds for the eigenvalue $\lambda_{c_i,c_b}$ in terms of the parameters $c_i$ and $c_b$.
\begin{proposition}\label{prop:lambda_cicb_est}
    Let $c_i,c_b\in\R$, and $\Om\subset\R^d$ be a  Lipschitz  bounded open set. Then
    \begin{itemize}
        \item If $-c_b< c_i$,
        \[
        -c_b< \lambda_{c_i,c_b}(\Om)<c_i.
        \]
        \item If $c_i<-c_b$, 
        \[
        c_i<\lambda_{c_i,c_b}(\Om)<-c_b.
        \]
    \end{itemize}
\end{proposition}
\begin{proof}
    \textbf{Case $-c_b< c_i$.} 
    
    For any $(u,v)\in {W^{1,2}}(\Om)\times {W^{1,2}}(\Sigma)$ we have
    \begin{equation}\label{eq:E_CiCb_below}
        \mathcal{E}^\Om_{c_i,c_b}(u,v) \geq \frac{\displaystyle c_i\int_\Omega u^2-c_b\int_\Sigma v^2}{\displaystyle \int_\Omega u^2+\int_\Sigma v^2}\geq-c_b,
    \end{equation}
since $c_i>-c_b$, so that
    \begin{equation}\nonumber
        \lambda_{c_i,c_b}(\Omega)\geq -c_b.
    \end{equation}
   See now that equality does not occur, 
    as otherwise $\mathcal{E}^\Om_{c_i,c_b}(u,v)=-c_b$ for the couple of principal eigenfunctions $(u,v)$, and, following the chain of inequalities in \eqref{eq:E_CiCb_below} it would imply that $u$ and $v$ are constants with $u_{|\Sigma}=v\ne0$, thus obtaining
    \begin{equation}
        \label{eq:below_lambdaOm_strict}
    -c_b=\lambda_{c_i,c_b}(\Omega)=\frac{\displaystyle c_i\int_\Omega u^2-c_b\int_\Sigma v^2}{\displaystyle \int_\Omega u^2+\int_\Sigma v^2}>-c_b,
       \end{equation}
which is a contradiction.

To prove the upper bound, observe that by \eqref{eq:lambda_variat} one has 
    \[
        \lambda_{c_i,c_b}(\Omega)\leq \mathcal{E}^\Om_{c_i,c_b}(u_0,v_0),
    \]
    where we let $u_0$ and $v_0$ be constant functions equal to $1$, thus getting 
    \[
        \lambda_{c_i,c_b}(\Omega)\leq \frac{\displaystyle c_i\int_\Omega u_0^2-c_b\int_\Sigma v_0^2}{\displaystyle \int_\Omega u_0^2+\int_\Sigma v_0^2}<c_i,
    \]
    since $-c_b<c_i$.

    \textbf{Case $c_i<-c_b$.}
    
     This case is treated similarly. 
\end{proof}

\subsection{Counterexample to global existence (Proof of Theorem \ref{Th:NonExistence})}
In this short paragraph we show that in general, one cannot expect existence to hold for the minimisation of $\lambda_{c_i,c_b}$ at fixed volume. 
\begin{proof}[Proof of Theorem \ref{Th:NonExistence}]
    Using Proposition \ref{prop:lambda_cicb_est} we have
    \[
    \lambda_{c_i,c_b}(\Om)>-c_b
    \]
    for any Lipschitz bounded open set $\Om$.
    On the other hand, by definition of $\lambda_{c_i,c_b}$ (see \eqref{eq:lambda_variat}) one has 
    \[
        \lambda_{c_i,c_b}(\Omega)\leq \mathcal{E}^\Om_{c_i,c_b}(u_0,v_0),
    \]
    where we have let $u_0$ and $v_0$ be constant functions equal to $1$, thus getting 
    \[
        \lambda_{c_i,c_b}(\Omega)\leq \frac{c_i|\B|-c_bP(\Om)}{|\B|+P(\Omega)},
    \]
    with $P(\Om)$ denoting the perimeter of $\Om$.
   Choosing now a sequence $\Om_j$ of Lipschitz sets with constant volume $|\B|$ verifying $P(\Omega_j)\to+\infty$ (taking for instance a sequence of thinning rectangles), we obtain that 
    \[
        \inf_{\Omega\in\mathcal{O}_{\text{c}}}\lambda_{c_i,c_b}(\Omega)\leq -c_b,
    \]
     Note that the same inequality for the class $\mathcal{O}_{\text{s}}$ can be shown, by (for instance) regularising a sequence of thinning rectangles. \BBB
    Together with the reverse inequality shown above, we thus have 
    \[
        \inf_{\Omega\in\mathcal{O}_{\text{c}}}\lambda_{c_i,c_b}(\Omega)= \inf_{\Omega\in\mathcal{O}_{\text{s}}}\lambda_{c_i,c_b}(\Omega)= -c_b.
    \]
    
    Note that there cannot be any Lipschitz $\Omega$ achieving this infimum, using again the strict lower bound from Proposition \ref{prop:lambda_cicb_est}. This finishes the proof.

\end{proof}
 
\subsection{Second-order analysis of the ball} \label{sect:second_order}

Let $\Om\subset\R^d$ be a bounded open set with a $\C^{2,\eta}$  boundary $\Sigma:=\partial\Om$, for some $\eta\in(0,1)$. In this section we  perform a second-order analysis of the functional 
\begin{equation}
    \label{eq:egv_functional}
    \theta\in\C^{2,\eta}(\R^d,\R^d)\mapsto\lambda_{c_i,c_b}(\Om_\theta)\in\R,
\end{equation}
where we let $\Om_\theta:=(\text{Id}+\theta)(\Om)$ be the $\C^{2,\eta}$ deformation of $\Om$ by the field $\theta\in\C^{2,\eta}(\R^d,\R^d)$. More precisely, after computing the derivative of \eqref{eq:egv_functional} at any shape $\Om$ (Proposition \ref{prop:1st_shape}) and proving that the ball is a critical point under volume constraint (Corollary \ref{cor:critpt}), we compute the second-order shape derivative at the ball (in Propositions \ref{prop:second_shaped} and \ref{prop:diag_lag}), on which we rely in order to exhibit regimes for which the ball  satisfies (or not) second-order optimality conditions (Theorem \ref{thm:local_ball}).  Throughout, we will thus be working with competitors of the form \begin{equation}
    \label{eq:perturbations_class}
\left\{\Om_\theta, \ \theta\in \C^{2,\eta}(\R^d,\R^d), \|\theta\|_{\C^{2,\eta}(\R^d,\R^d)}\ll1\right\} 
\end{equation}
where $\eta\in(0,1)$ is  arbitrary (see Remark \ref{rk:locmin} for the link with the $\C^{2,\eta}$-local minimality of the ball). 

 Let us say that \cite{DKL16} is an important reference for this section, where the authors carry out an analogous second-order analysis for the Wentzell-Laplace problem, and prove that the ball verifies second-order optimality conditions \cite[Corollary 1.8]{DKL16}. 

\textbf{Notations.}

 We denote by $\nu_\Sigma$, (or more simply $\nu$ when no confusion is possible) the outer  normal vector to $\Sigma$, and by $\partial_\nu$, $\parnunu$ and so on, the successive derivatives in the direction $\nu_\Sigma$. For differential quantities associated to $\Sigma$ we add a $\Sigma$ index: $\nabsig$, $\text{div}_\Sigma$, $\Delta_\Sigma$ respectively denote the tangential gradient, divergence and Laplacian. 

In this section, we generally drop the dependence of the eigenvalue $\lambda$ on the parameters $c_i, c_b\in\R$ and write  $\lambda(\Om)$ for $\lambda_{c_i,c_b}(\Om)$. We define $\lambar$ and $\lamti$ as
\begin{equation}
    \label{eq:def_lambdatilde_bar}
\lambar(\Om)=\lambda(\Om)-c_i, \ \lamti(\Om):=\lambda(\Om)+c_b-1
\end{equation}
so that the system verified by the eigencouple $(u,v)$ associated to $\lambda(\Om)$ simply writes
\begin{equation}
    \label{eq:principalegf_syst}
\begin{cases}
    -\Delta u=\lambar u & \text{in }\Omega,\\
    -\Delta_{\Sigma} v=\lamti v+u &\text{over }\Sigma,\\
    \partial_{\nu} u+u=v &\text{over }\Sigma,\\
    u>0 \text{ in }\Om,\ v>0 \text{ over }\Sigma,\\
    \int_\Om u^2+\int_\Sigma v^2=1.
\end{cases}
\end{equation}

For any smooth path of $\C^{2,\eta}$ vector fields $t\mapsto X_t\in\C^{2,\eta}(\R^d,\R^d)$ verifying $X_0=\text{Id}$, we let 
\[
\Om_t:=X_t(\Om),\ \Sigma_t:=\partial\Om_t
\]
which is a bounded open Lipschitz set for $|t|\ll1$ (whenever $\|X_t-\text{Id}\|_{W^{1,\infty}(\R^d,\R^d)}<1$). Let us emphasise that these paths of deformations of $\Om$ encompass the central example
\[
\forall |t|\ll1, \ X_t:=\text{Id}+t\th,
\]
where $\th\in\C^{2,\eta}(\R^d,\R^d)$.

For any  small $t$, we let $(u_t,v_t)\in \C^{2,\eta}(\ov{\Om_t})\times \C^{2,\eta}(\Sigma_t)$ be the couple of principal eigenfunction of $\lambda_t:=\lambda(\Om_t)$, \textit{i.e.}
\begin{equation}
    \label{eq:ut_vt_system}
\begin{cases}
    -\Delta u_t=\lambar_t u_t & \text{in }\Omega_t,\\
    -\Delta_{\Sigma_t} v_t=\lamti_tv_t+u_t &\text{over }\Sigma_t,\\
    \partial_{\nu_t} u_t+u_t=v_t &\text{over }\Sigma_t,\\
    u_t>0 \text{ in }\Om_t,\ v_t>0 \text{ over }\Sigma_t,\\
     \int_{\Om_t} u_t^2+\int_{\Sigma_t} v_t^2=1.
\end{cases}
\end{equation}
with $\lambar_t:=\lambar(\Om_t)$ and $\lamti_t:=\lamti(\Om_t)$.  In order to differentiate \eqref{eq:egv_functional} and obtain the expressions of $\lambda_t'$ and $\lambda_t''$, we will differentiate the above system \eqref{eq:ut_vt_system}. As the natural functional space $\C^{2,\eta}(\ov{\Om_t})\times \C^{2,\eta}(\Sigma_t)$  varies with $t$, we first find a common functional space where we can carry out the differentiation of the maps $t\mapsto u_t$ and $t\mapsto v_t$. The same goes for the outer unit normal $\nu_t\in \C^{1,\eta}(\Sigma_t)$. This is done through particular choices of extensions of $u_t$, $v_t$, as well as of $\nu_t$, although we emphasise the fact that these choices do not appear in the final formulae of the derivatives (see below Propositions \ref{prop:1st_shape} and \ref{prop:diag_lag}).

\textbf{Shape calculus.}

Let us recall some standard notions of shape  calculus (for a detailed introduction we refer to \cite[Chapter 5]{HP18}). The main tool for differentiating shape functionals are the two Hadamard formulae:
\begin{equation}
    \label{eq:had_domain}
\frac{d}{dt}\left(\int_{\Om_t}f_t\right)_{|t=0}=\int_{\Om}f_0'+\int_\Sigma f_0\theta_\nu,
\end{equation}
\begin{equation}
    \label{eq:had_boundary}
\frac{d}{dt}\left(\int_{\Sigma_t}g_t\right)_{|t=0}=\int_{\Sigma}g_0'+\int_\Sigma \theta_\nu\left(\parnu g_0+Hg_0\right),
\end{equation}
where $f_t:\Om_t\to\R$, $g_t:\Sigma_t\to\R$ satisfy sufficient regularity and differentiability assumptions (see respectively \cite[Theorem 5.2.2]{HP18} and \cite[Proposition 5.4.18]{HP18}), with the notations 

\[
f_0':=\frac{d}{dt}(f_t)_{|t=0},\  g_0':=\frac{d}{dt}(g_t)_{|t=0},\ \theta:=X_0'=\frac{d}{dt}(X_t)_{t=0},
\]
$\theta_\nu:=\theta\cdot\nu$ is the normal component of $\theta$, and $H$ is the mean curvature of $\Sigma$.

As usual when performing differentiation of boundary integrals, it is convenient to extend the outer unit normal $\nu_\Sigma:\Sigma\to\R$ to the whole $\R^d$ by setting:
\begin{equation}
    \label{eq:ext_nu}
\overline{\nu}:=
    \chi\nabla b_{\Sigma},\\
\end{equation}
where $b_\Sigma$ is the signed distance function to $\Sigma$, and $\chi:\R^d\to\R$ is a smooth cut-off function with $\chi\equiv1$ in a neighborhood of $\Sigma$. Still denoting $\nu$ this extension, we have that (see for instance \cite[Appendix B.3]{CDK13} and \cite[Proposition 5.4.14]{HP18})
\begin{equation}
    \label{eq:Dnu_symm}
    D\nu=D^2b_\Sigma \text{ is symmetric and }D\nu\cdot\nu=0 \text{ on }\Sigma.
    \end{equation}
Denoting by $\nu_t$ the outer unit normal to $\Sigma_t$ (extended to $\R^d$), the shape derivative of $\nu_t$ has the simple form (see again \cite{CDK13,HP18})
\begin{equation}\label{eq:nu'}
\nu':=\frac{d}{dt}(\nu_t)_{|t=0}=-\nabsig\theta_\nu.
\end{equation}

\textbf{Shape differentiation of eigenelements.} 

To differentiate the eigenfunctions, we first extend them to $\R^d$ in the following way (see \cite[Section 5.5]{HP18}). First, the regularity of $\Om$ allows to find extension mappings $P^1_\Om$ and $P^2_\Om$
\[
u\in {W^{1,2}}(\Om)\mapsto P^1_\Om u\in {W^{1,2}}(\R^d),\]
\[v\in {W^{1,2}}(\Sigma)\mapsto P^2_\Om v\in {W^{1,2}}(\R^d),
\]
where in addition we impose that $P^2_\Om$ be built so that $P^2_\Om v:\R^d\to\R$ be constant along the normal $\nu$ in a neighbourhood of $\Sigma$.
We  then let, for each small $t$, $\phi_t$ be the $\C^{2,\eta}$ diffeomorphism
\[
\phi_t:=\text{Id}+t\theta,
\]
and set 
\begin{equation}
    \label{eq:def_ext_egf}
\begin{cases}
    \overline{u}_t:=P^1_\Om(u_t\circ\phi_t)\circ\phi_t^{-1}\in {\C^{2,\eta}}(\R^d),\\
     \overline{v}_t:=P^2_\Om(v_t\circ\phi_t)\circ\phi_t^{-1}\in {\C^{2,\eta}}(\R^d).
\end{cases}
\end{equation}
For the sake of notational simplicity we keep writing $u_t$ and $v_t$ for these extended eigenfunctions. The eigencouple and the eigenvalue are twice continuously differentiable, \textit{i.e.} the mapping
\[
t\mapsto (u_t,v_t, \lambda_t)\in \C^{0,\eta}(\R^d)\times
\C^{0,\eta}(\R^d)\times\R\]
 is  $\C^2$ around $t=0$ (the eigenelements are even $\C^2$ in the sense of Fréchet derivatives; although the method of proof is classical, we detail these considerations in Appendix \ref{Ap:Differentiability}).  We  denote by $(u_t',v_t',\lambda'_t)$ (respectively $(u_t'',v_t'',\lambda''_t)$) its first (respectively second) derivative. 

\textbf{Tangential differentiation.}

Let us recall the useful integration by parts formulae
    \begin{equation}
        \label{eq:ibp_tang1}
        \int_\Sigma \nabsig f\cdot\nabsig g=-\int_\Sigma g\Delta_\Sigma f,
    \end{equation}
    and 
    \begin{equation}
        \label{eq:ibp_tang2}
        \int_\Sigma \text{div}_\Sigma(W)f=-\int_\Sigma W\cdot \nabsig f,
    \end{equation}
    for a smooth enough tangential (\textit{i.e.} verifying $W_\nu:=W\cdot\nu=0$ over $\Sigma$) vector field $W:\R^d\to\R^d$ and sufficiently smooth functions $f,g:\R^d\to\R$ (see for instance \cite[Proposition 5.4.9]{HP18}).
We also remind the classical decomposition of the Laplacian on $\Sigma$ (see \cite[Proposition 5.4.12]{HP18}):
    \begin{equation}
        \label{eq:lap_tangential}
        \Delta u=\Delta_\Sigma u+H\parnu u +\parnunu u.
    \end{equation}
    
    \textbf{First-order derivative.}
    
In the next proposition we compute
\[
\frac{d}{dt}\lambda(\Om_t)_{|t=0},
\]
where $t\mapsto X_t\in \C^{2,\eta}(\R^d,\R^d)$ is a smooth path of $\C^{2,\eta}$ vector fields.

\begin{proposition}[First shape derivative]\label{prop:1st_shape}
    Let $\Om\subset\R^d$ be a bounded open set with  $\C^{2,\eta}$  boundary $\Sigma:=\partial\Om$ and principal eigencouple $(u,v)$ verifying \eqref{eq:principalegf_syst}. Let $t\mapsto X_t\in\C^{2,\eta}(\R^d,\R^d)$ be any smooth path of $\C^{2,\eta}$ vector fields verifying $X_0=\text{Id}$. Setting $\theta:=X_0':=\frac{d}{dt}(X_t)_{t=0}$, it holds
    \begin{align}\nonumber
        \frac{d}{dt}\Big(\lambda\left(X_t(\Om)\right)\Big)_{|t=0}= \int_\Sigma \theta_\nu\Big(|\nabla_\Sigma u|^2-(\partial_{\nu} u)^2-\langle B\nabla_\Sigma v,\nabla_\Sigma v\rangle\\-\lambar u^2+H(u^2-2uv-\lamti v^2)\Big), \label{eq:1st_shape_d_express}
    \end{align}
    where $B:=2D^2b-H\text{Id}$, $b:\R^d\to\R^d$ being the signed distance function to $\Sigma$ and $H$ its mean curvature.
\end{proposition}

\begin{remark}
    From  \eqref{eq:1st_shape_d_express}
 one  recovers the expression of the shape derivative of $\lambda$ at $\Om$, given by the Fréchet derivative at $0$ of 
 \[\th\in\C^{2,\eta}(\R^d,\R^d)\mapsto \lambda((\text{Id}+\th)(\Om)).
 \]
 
 In fact, taking $X_t:=\text{Id}+t\th$, the shape derivative of $\lambda$ at $\Om$ in the direction $\th$ writes
 
\begin{equation}
    \label{eq:def_shape_d}
\lambda'(\Om)\cdot(\th)=
        \frac{d}{dt}\Big(\lambda\left(X_t(\Om)\right)\Big)_{|t=0}, 
    \end{equation}
    which is thus also given by expression \eqref{eq:1st_shape_d_express}.
 \end{remark}

\begin{proof}[Proof of Proposition \ref{prop:1st_shape}]
    We start by differentiating the system \eqref{eq:ut_vt_system}, where $(u_t,v_t,\lambda_t)$ still denote the eigenelements on $\Om_t=X_t(\Om)$. First, using  \cite[Theorem 5.5.2]{HP18}, we have
    \begin{equation}
        \label{eq:syst_u'}
    \begin{cases}
        -\Delta u'=\overline{\lambda}u'+\lambda' u, & \text{in }\Om\\
        \partial_\nu u'=\theta_\nu\big(\parnu(v-u)-\parnunu u\big)+\nabsig u\cdot\nabsig \theta_\nu+(v'-u') & \text{over }\Sigma,
    \end{cases}
    \end{equation}
    where $(u',v',\lambda')$ denotes the derivative $\frac{d}{dt}(u_t,v_t,\lambda_t)_{t=0}$ (note that in comparison with \cite[Theorem 5.5.2]{HP18} there is an additional term due to the differentiation of the boundary datum $v_t-u_t$, which is handled through the Hadamard  formula \eqref{eq:had_boundary}). We  simply note that this is obtained through the differentiation of \eqref{eq:ut_vt_system} written variationally, relying on the Hadamard formulae \eqref{eq:had_domain} and \eqref{eq:had_boundary}.
    
    As for the differentiation of the equation verified by $v_t$, we have
    \begin{align}
        \nonumber
        -\Delta_\Sigma v'-\Delta_\Sigma(\theta_\nu\partial_\nu v)&+\text{div}_\Sigma(\theta_\nu B\nabsig v)
        =\\
        \lamti v'+u'  &+\lambda' v+\theta_\nu\big(\partial_\nu(\tilde{\lambda}v+u)+H(\lamti v+u)\big)  \qquad \text{ over }\Sigma,\label{eq:v'eq}
    \end{align}
    for which we refer to \cite[Proof of Proposition 5.1]{CDK13}. More precisely, picking a test function $\varphi\in {W^{1,2}}(\Om)$ (extended to $\varphi\in {W^{1,2}}(\R^d)$), applying \eqref{eq:had_boundary} one finds 
        \begin{align}
    \nonumber
        \frac{d}{dt}\left(\int_{\Sigma_t}\nabsigt v_t\nabsigt \varphi\right)_{|t=0}=\int_\Sigma&\left(-\Delta_\Sigma v'-\Delta_\Sigma(\theta_\nu\partial_\nu v)+\text{div}_\Sigma(\theta_\nu B\nabsig v)\right)\varphi \\
         & -\int_\Sigma \left(\th_\nu\Delta_\Sigma v\right)\parnu\varphi,
        \label{eq:nabv_nabphi_diff}
    \end{align}
    while on the other hand
    \begin{align*}
    \frac{d}{dt}\left(\int_{\Sigma_t} \left(\lamti_t v_t+u_t\right)\varphi\right)_{|t=0}=\int_\Sigma&\left(\lamti v'+u'+\lambda' v+\theta_\nu\big(\partial_\nu(\tilde{\lambda}v+u)+H(\lamti v+u)\big)\right) \varphi\\
    &+\int_\Sigma\th_\nu \left(\lamti v+u\right)\parnu \varphi.
    \end{align*}
     Recalling that $-\Delta_\Sigma v=\lamti v+u$, we thus deduce \eqref{eq:v'eq}.
        
    Multiplying by $u$ the equation verified by $u'$ in \eqref{eq:syst_u'} we have
    \begin{equation*}
    -\int_\Om u\Delta u'=\overline{\lambda}\int_\Om u'u+\lambda'\int_\Om u^2.
    \end{equation*}
    Integrating by parts and using that $-\Delta u=\overline{\lambda}u$ inside $\Om$ one has
    \[
    -\int_\Om u\Delta u'=\overline{\lambda}\int_\Om u'u+\int_\Sigma(u'\parnu u-\parnu u'u),
    \]
    thus providing
    \begin{equation*}
        \lambda'\int_\Om u^2=\int_\Sigma(u'\parnu u-\parnu u'u).
    \end{equation*}
    We rewrite the previous using the expressions for $\parnu u$ and $\parnu u'$ from \eqref{eq:principalegf_syst} and \eqref{eq:syst_u'}, which gives
    \begin{equation}
        \label{lam'intu2}
        \lambda'\int_\Om u^2=\int_\Sigma u'(v-u)-u\Bigl(\theta_\nu \bigl( \parnu(v-u)-\parnunu u\bigr) +\nabsig \theta_\nu\nabsig u+(v'-u') \Bigr).
    \end{equation}
    Let us now multiply by $v$ the equation verified by $v'$ and integrate by parts the left-hand-side of \eqref{eq:v'eq}. Using \eqref{eq:ibp_tang1}, \eqref{eq:ibp_tang2} gives    \begin{align*}
        \int_\Sigma\Big(-\Delta_\Sigma v'-\Delta_\Sigma(\theta_\nu\partial_\nu v)&+\text{div}_\Sigma(\theta_\nu B\nabsig v)\Big)v \\&= \int_\Sigma \Big((v'+\theta_\nu\parnu v)(\lamti v+u)-\theta_\nu \langle B\nabsig v,\nabsig v\rangle \Big)
    \end{align*}
    where we also used that $-\Delta_\Sigma v=\lamti v+u$ over $\Sigma$ and the fact that 
    \[B\nabsig v\cdot\nu=\nabsig v\cdot B\nu=0 \text{ over }\Sigma,\] since $B$ is symmetric and $\nu$ is tangential over $\Sigma$ (see \eqref{eq:Dnu_symm}). We thus obtain
    \begin{align}
    \nonumber
        \int_\Sigma\Big(v'u+\theta_\nu u\parnu v-\theta_\nu \langle B\nabsig v,\nabsig v\rangle\Big)&=\int_\Sigma \Big(\lambda'v^2+u'v\\&+\theta_\nu v\left(\parnu u+H(\lamti v+u)\right)\Big).
        \label{eq:interm_v'}
    \end{align}
Summing \eqref{lam'intu2}--\eqref{eq:interm_v'} yields
    \begin{align*}
        \lambda'+&\int_\Sigma\Big(u'v+\theta_\nu v\big(\parnu u+H(\lamti v+u)\big)\Big)=\int_\Sigma \Big(v'u+\theta_\nu u\parnu v-\theta_\nu \langle B\nabsig v,\nabsig v\rangle\Big)\\
        &+\int_\Sigma \biggl( u'(v-u)-u \Bigl( \theta_\nu \bigl( \parnu(v-u)-\parnunu u \bigr) +\nabsig \theta_\nu\nabsig u+(v'-u') \Bigr) \biggr).
    \end{align*}
    Simplifying the terms $\int_\Sigma u'v$, $\int_\Sigma u'u$, $\int_\Sigma uv'$ and $\int_\Sigma \theta_\nu u\parnu v$ we get
    \begin{align}
        \nonumber
        \lambda'+\int_\Sigma\Big(\theta_\nu v\big(\parnu u+H(\lamti v+u)\big)\Big)&=-\int_\Sigma \theta_\nu \langle B\nabsig v,\nabsig v\rangle\\
        &+\int_\Sigma \biggl( u\Bigl( \theta_\nu(\parnu u+\parnunu u)-\nabsig \theta_\nu\nabsig u \Bigr) \bigg),
        \label{eq:interm_lambda'}
    \end{align}
    We now integrate by parts the term $\int_\Sigma u \nabsig \theta_\nu\cdot\nabsig u$ using \eqref{eq:ibp_tang2}:
    \begin{align*}
        \int_\Sigma u \nabsig \theta_\nu\cdot \nabsig u &=-\int_\Sigma \theta_\nu\text{div}_{\Sigma}(u\nabsig u)\\
        &=-\int_\Sigma \theta_\nu\big(u\Delta_\Sigma u+|\nabsig u|^2\big)\\
        &= \int_\Sigma \theta_\nu\Big(u\big(\overline{\lambda}u+H\parnu u+\parnunu u\big)-|\nabsig u|^2\Big)
    \end{align*}
    where we used \eqref{eq:lap_tangential} in the fourth line. We thus rewrite the right hand-side of \eqref{eq:interm_lambda'}
    \begin{align*}
        &\int_\Sigma \theta_\nu \biggl( - \langle B\nabsig v,\nabsig v\rangle+ u(\parnu u+\parnunu u)-\Bigl(  u\bigl(  \overline{\lambda}u+H\parnu u+\parnunu u \bigr)-|\nabsig u|^2 \Bigr) \biggr)\\
        &=\int_\Sigma \theta_\nu \biggl(  - \langle B\nabsig v,\nabsig v\rangle+ u\parnu u- \Bigl( u\bigl( \overline{\lambda}u+H\parnu u\bigr)-\abs{\nabsig u}^2 \Bigr) \biggr),
    \end{align*}
    Recalling that $\parnu u=v-u$ over $\Sigma$, we finally get 
    \[
        \lambda'=\int_\Sigma \theta_\nu\Big( \abs{\nabsig u}^2-(\parnu u)^2- \langle B\nabsig v,\nabsig v\rangle-u^2\overline{\lambda}+H\big(u^2-2uv-\lamti v^2\big)\Big),
    \]
    which is the announced expression for $\lambda'$, thus finishing the proof.

\end{proof}

From the expression of the first order shape derivative $\lambda'(\B)$ given by \eqref{eq:1st_shape_d_express} (thanks also to the remark made in \eqref{eq:def_shape_d}), one quickly deduces that the ball is a critical point of $\lambda$ (under volume constraint), as is stated in next result. Denoting by $\mathrm{Vol}$ the volume functional, recall that we write $\mathrm{Vol}'$ its first shape derivative, which is classicaly given by (see for instance \cite[Theorem 5.2.2]{HP18}, or alternatively apply the Hadamard formula \eqref{eq:had_domain}):
\[
\forall \theta\in \C^{2,\eta}(\R^d,\R^d), \mathrm{Vol}'(\Om)\cdot(\theta)=\int_{\Sigma}\theta_\nu
\]

\begin{corollary}[The ball is a critical point]\label{cor:critpt}
 The principal eigencouple $(u,v)$ of the ball (verifying \eqref{eq:principalegf_syst} for $\Om=\B$) is such that
    \[\begin{cases}
        u \text{ is radial in } \B,\\ 
        v \text{ is constant on } \S,
    \end{cases}\]
    and we have, on $\S$,
    \[
    \begin{cases}
        u=-\lamti v,\\
        \parnu u=v(\lamti+1).
    \end{cases}
    \]
    This implies that the ball $\B$ is a critical point of $\lambda$ under volume constraint, meaning that
    \[ \forall \theta\in \C^{2,\eta}(\R^d,\R^d)\,, \int_{\S}\theta_\nu=0\Rightarrow \lambda'(\B)\cdot (\theta)=0\]
    or equivalently, there exists a Lagrange multiplier $\mu\in\R$ such that 
    \[\lambda'(\B)=\mu Vol'(\B).\]
        \end{corollary}

\begin{proof}
    As $(u,v)$ is a principal eigencouple then by definition
    \[\lambda(\B)=\mathcal{E}^\B_{c_i,c_b}(u,v),\]
    (see \eqref{eq:lambda_variat}). Let $R$ be any rotation $R:\R^d\to\R^d$. Setting $u_R:=u(R\cdot)$ and $v_R:=v(R\cdot)$, one notices that $\mathcal{E}^\B_{c_i,c_b}(u_R,v_R)=\mathcal{E}^\B_{c_i,c_b}(u,v)$. Since we also have $\|u_R\|_{L^2(\B)}=\|u\|_{L^2(\B)}$ and $\|v_R\|_{L^2(\S)}=\|v\|_{L^2(\S)}$, then by uniqueness of a normalized solution to \eqref{eq:principalegf_syst} we deduce 
    \[\begin{cases}
        u_R=u,\\
        v_R=v.
    \end{cases}\]
As a consequence $u$ is radial on $\B$ and $v$ is constant on $\S$. Plugging this piece of information into the expression \eqref{eq:1st_shape_d_express}, there exists a constant $\mu\in\R$ such that 
\[\forall \theta\in \C^{2,\eta}(\R^d,\R^d),\  
\lambda'(\B)\cdot(\theta)=\mu\int_{\S} \theta_\nu.
\]
Let us finally express $u$ in terms of $v$: using the boundary condition for $u$ and the equation verified by $v$ in \eqref{eq:system_egv} we have 
    \[\begin{cases}
    0= \lamti v+u &\text{on } \S,\\
    \parnu u+u=v &\text{on }\S,
    \end{cases}
\]
providing the announced expressions on $\S$
\[
    \begin{cases}
        u=-\lamti v,\\
        \parnu u=v(\lamti+1).
    \end{cases}
    \]
\end{proof}

\textbf{Second-order derivative.}

From now on, the analysis focuses on the ball $\B$. We first compute the second-shape derivative of $\lambda$ at the ball. Following \cite{DKL16}, in next result we perform the computations along what we will call second-order volume-preserving paths, \textit{i.e.} smooth paths $t\mapsto X_t\in\C^{2,\eta}(\R^d,\R^d)$ of $\C^{2,\eta}$ vector fields verifying 
\[
\frac{d}{dt}\Big(\mathrm{Vol}(X_t(\B))\Big)_{|t=0}=0 \text{ and } \frac{d^2}{dt^2}\Big(\mathrm{Vol}(X_t(\B))\Big)_{|t=0}=0.
\]
Using the Hadamard  formulae \eqref{eq:had_domain} and \eqref{eq:had_boundary}, these conditions rewrite
\begin{equation}
    \label{eq:1st_2nd_vol_preserv}
\int_\S\th_\nu=0 \text{ and }
\int_\S\left(\frac{d}{dt}\left(X_t'\cdot\nu_{X_t}\right)_{t=0}+\th_\nu\left(\parnu \th_\nu+H\th_\nu\right)\right)=0,
\end{equation}
where $\nu_{X_t}$ is the outer unit normal of $X_t(\B)$ and we have set $\theta:=X_0'=\frac{d}{dt}(X_t)_{|t=0}$. In the following proposition we derive the expression of $\frac{d^2}{dt^2}\lambda(X_t(\B))_{|t=0}$ for such paths. 

\begin{proposition}[Second shape derivative at the ball]\label{prop:second_shaped}
Let $t\mapsto X_t\in \C^{2,\eta}(\R^d,\R^d)$ be a smooth path of $\C^{2,\eta}$ vector fields preserving volume at second-order, meaning that 
\[
\frac{d}{dt}\Big(\mathrm{Vol}(X_t(\B))\Big)_{|t=0}=\frac{d^2}{dt^2}\Big(\mathrm{Vol}(X_t(\B))\Big)_{|t=0}=0.
\]
Then, denoting $\theta:=\frac{d}{dt}(X_t)_{|t=0}=X_0'$, we have
\[
    \frac{d^2}{dt^2}\Big(\lambda\left(X_t(\B)\right)\Big)_{|t=0}=\int_{\S}\alpha \theta_\nu^2+\int_{\S}\beta|\nabla_\Sigma \theta_\nu|^2+\int_{\S} \theta_\nu\Big(\gamma v'+\delta u'\Big),
\]
where \[
\begin{cases}
    \alpha:=\parnu u\left(-2(H-1)\parnu u+u(H-2\lambar) \right),\\
    \beta:=-u(\parnu u),\\
\gamma:=-2(\parnu u),\\
\delta:=-2\left((H-1)\parnu u+\lambar u\right).
\end{cases}
\] 
and we have set $(u',v'):=\frac{d}{dt}(u_t,v_t)_{t=0}$ with $(u_t,v_t)$ a principal eigencouple of $X_t(\B)$ (verifying \eqref{eq:ut_vt_system}), and $(u,v):=(u_0,v_0)$.
\end{proposition}
\begin{proof}
    Let us set $\B_t:=X_t(\B)$, $\Sigma_t:=\partial\B_t$.
    We first differentiate the system verified by $u_t',v_t'$ respectively verifying (see \eqref{eq:syst_u'} and \eqref{eq:v'eq}): 
    \begin{align}
    \begin{cases}
        -\Delta u_t'=\overline{\lambda_t}u_t'+\lambda_t' u_t & \text{in }\B_t,\\
        \partial_{\nu_t} u_t'=\th^t_{\nu_t}\big(\parnut(v_t-u_t)-\parnunut u_t\big)+\nabsigt u_t\nabsigt \th^t_{\nu_t}+(v_t'-u_t') & \text{over }\Sigma_t,
    \end{cases}
    \label{eq:u_t'}
    \end{align}
    and 
    \begin{align}\nonumber
        -\Delta_{\Sigma_t} v_t'-\Delta_{\Sigma_t}(\th^t_{\nu_t}\parnut v_t)+&\text{div}_{\Sigma_t}(\th^t_{\nu_t} B_t\nabsigt v_t)
        =
        \lamti_t v_t'+u_t'\\&+\lambda_t' v_t+\th^t_{\nu_t}\big(\parnut(\tilde{\lambda_t}v_t+u_t)+H_t(\lamti_t v_t+u_t)\big)  \qquad \text{ over }\Sigma_t,
        \label{eq:LHSSecondderiv}
    \end{align}
      where we have set $\th^t:=X_t'$ (and thus $\th^t_{\nu_t}:=X_t'\cdot\nu_t$), $B_t:=2D^2b_t-H_t\text{Id}$ with $b_t:\R^d\to\R^d$ denoting the signed distance function to $\Sigma_t$ and $H_t$ its mean curvature.
In the proof, the notations will omit the references to $t=0$, meaning that rather than writing $u_0'$, $u_0''$ we will write more simply $u'$, $u''$, and so on. 

\textbf{Differentiation of $v_t'$.} 

On the other hand, we differentiate the equation verified by $v_t'$ using the computations from \cite[Proof of Theorem 3.5]{DKL16}. Taking $v:\S\to \R$ as a test function (extended to $\R^d$ by \eqref{eq:def_ext_egf}) it holds (see \eqref{eq:nabv_nabphi_diff}):

\begin{align*}
    \frac{d}{dt}\left(\int_{\Sigma_t}\nabsigt v_t'\nabsigt v\right)_{t=0}=\int_\S &\Big(\nabsig v''+\nabsig(\th_\nu\parnu v')-\th_\nu (B\nabsig v')\Big)\cdot\nabsig v\\&  -\int_\Sigma \left(\th_\nu\Delta_\Sigma v'\right)\parnu v.\end{align*}
Recalling that $v$ is constant on $\S$ (thanks to Corollary \ref{cor:critpt}) we have $\nabsig v=0$, while the choice of extension of $v$ (see \eqref{eq:def_ext_egf}) ensures that $\parnu v=0$, thus yielding
\[\frac{d}{dt}\left(\int_{\Sigma_t}\nabsigt v_t'\nabsigt v\right)_{t=0}=0.\]
With the same kind of considerations we likewise get 
\[
\frac{d}{dt}\left(\int_{\Sigma_t}\nabsigt (\th^t_{\nu_t}\parnut v_t)\nabsigt v\right)_{t=0}=0
\]
and
\[
\frac{d}{dt} \left(\int_{\Sigma_t}\th^t_{\nu_t}\langle B_t\nabsigt v_t,\nabsigt v\rangle\right)_{t=0}=0.
\]
After multiplying \eqref{eq:LHSSecondderiv} by $v$ and integrating by parts, we have
\begin{equation}
    \label{eq:beforedif_vt'}
\int_{\Sigma_t}\Big(\nabsigt v_t'+\nabsigt (\th^t_{\nu_t}\parnut v_t)-\th^t_{\nu_t} B_t\nabsigt v_t\Big)\cdot \nabsigt v= \int_{\Sigma_t}\tilde{g}_tv
\end{equation}
where $\tilde{g}_t:=\lamti_t v_t'+u_t'+\lambda_t' v_t+\th^t_{\nu_t}\big(\parnut(\tilde{\lambda_t}v_t+u_t)+H_t(\lamti_t v_t+u_t)\big)$. As a consequence of the three derivatives above, the derivative of the left-hand-side vanishes. On the other hand,  the Hadamard  formula \eqref{eq:had_boundary} yields
\[\frac{d}{dt}\left(\int_{\Sigma_t}\tilde{g}_tv\right)_{t=0}=\int_\S fv,\] with
\begin{align*}
f&:=\left[\lambda''v+\lamti v''+u''+\th_\nu'(\parnu w+Hw)+\th_\nu\left((\parnu w)'+H'w +Hw'\right)\right]\\
&+\th_\nu\left[\parnu \left((\lamti v'+u')+\th_\nu(\parnu w+Hw)\right)+H\left((\lamti v'+u')+\th_\nu(\parnu w+Hw)\right)\right],
\end{align*}
where we denote abusively $\theta_\nu':=\frac{d}{dt}(\th^t_{\nu_t})_{t=0}$ and set $w:=\lamti v+u$, and we have used that $\lambda'=0$ since $X_t$ is volume-preserving (see \eqref{eq:1st_2nd_vol_preserv}). 


We now simplify the expression of $f$. Using $\nabsig w=\nabsig (\lamti v+u)=0$ due to the specific structure of the eigencouple $(u,v)$ (see Corollary \ref{cor:critpt}), the differentiation $(\parnu w)'$ gives 
\begin{align}\nonumber
(\parnu w)':=\frac{d}{dt}(\parnut (\lamti_t v_t+u_t))_{t=0}&=\parnu(\lamti v+u)'+\nabla w\cdot \nu'\\\nonumber
&= \parnu (\lamti u'+v')-\nabsig w\cdot \nabsig \th_\nu\\
&=\parnu (\lamti u'+v')\label{eq:parnuw'}
\end{align}
recalling that $\lambda'=0$, and where we also used \eqref{eq:nu'}.  As $v$ is constant, $w=-\Delta_\Sigma v=0$ over $\Sigma$. By 
gathering the terms in $u,v$ the ones in $u', v'$ and the ones in $u'', v''$ together, one can rewrite $f$ as the sum of $Z_1:=Z_1(u'',v'')$, $Z_2:=Z_2(u',v')$, $Z_3:=Z_3(u,v)$, where 
\begin{align*}
    Z_1&=\lamti v''+u'',\\
    Z_2&=2\th_\nu\left(\parnu (\lamti v'+u')+H(\lamti v'+u')\right),\\
    Z_3&= \lambda''v 
    +(\parnu w+Hw)(\th_\nu'+\th_\nu\parnu \th_\nu+H\th_\nu^2)
    +\th_\nu^2(\parnunu w+H\parnu w)
\end{align*}
The differentiation of \eqref{eq:beforedif_vt'} thus yields
\begin{align}
0=\int_\S (Z_1+Z_2+Z_3)v
&=\int_\S (\lamti v''+u'')v+\int_\S\left[2\th_\nu\left(\parnu (\lamti v'+u')+H(\lamti v'+u')\right)\right]v\nonumber\\
&+\lambda''\int_\S v^2+\int_\S \th_\nu^2(\parnunu w+H\parnu w)v.\label{eq:lambda''_v2}
\end{align}
where we used that
\[
\int_\S (\parnu w+Hw)(\th_\nu'+\th_\nu\parnu \th_\nu+H\th_\nu^2)v=0
\]
since $u,v, \parnu u,\parnu v$ are constant over $\S$ (thanks to Corollary \ref{cor:critpt} and the extension of $v$, see \eqref{eq:def_ext_egf}), and $X_t$ is volume-preserving at second-order (see \eqref{eq:1st_2nd_vol_preserv}).

\textbf{Differentiation of $u_t'$.} 

 To differentiate the equation satisfied  by $u_t'$ together with its boundary condition, we refer to \cite[Theorem 5.5.2]{HP18}. We obtain
\begin{equation}
\begin{cases}
    
-\Delta u''=\lambda''u+2\lambda'u'+\lambar u'', & \text{ in } \B,\\
\parnu u''=\th_\nu(\parnu g-\parnunu u')+\nabsig u'\nabsig \th_\nu+g', &\text{ over }\S,
\end{cases}
 \label{eq:2ndderiv_u'}
\end{equation}
where we have set $g_t:=\th^t_{\nu_t}\big(\parnut(v_t-u_t)-\parnunut u_t\big)+\nabsigt u_t\nabsigt \th^t_{\nu_t}+(v_t'-u_t')$.


Let us start by expliciting the boundary condition in  \eqref{eq:2ndderiv_u'}. One has 
\begin{align*}
    \parnu g=\parnu \th_\nu(\parnu (v-u)-\parnunu u)+\th_\nu\parnu(\parnu (v-u)-\parnunu u)\\+\parnu(\nabsig u\cdot \nabsig \th_\nu)+\parnu(v'-u')
\end{align*}
But from \cite[Appendix A.3]{DKL16}, 
\[\parnu(\nabsig u\cdot \nabsig \th_\nu)=\nabsig (\parnu u)\cdot \nabsig \th_\nu+\nabsig  u\cdot \nabsig (\parnu \th_\nu)-2(D^2b\nabsig u)\cdot \nabsig \th_\nu.\]
Since $u$ is radial, $\nabsig u=\nabsig (\parnu u)=0$, so that $\parnu(\nabsig u\cdot \nabsig \th_\nu)=0$. On the other hand, we rewrite $g'$:
\[
    g'= \th_\nu' \bigl( \parnu (v-u)-\parnunu u \bigr)+\th_\nu \Bigl( \bigl( \parnu (v-u) \bigr)'-(\parnunu u)' \Bigr)+(\nabsig u\cdot\nabsig \th_\nu)'+(v''-u'')
\]

Now, as in \eqref{eq:parnuw'}, we have $(\parnu(v-u))'=\parnu (v-u)'$. On the other hand, 
\begin{align*}
    (\parnunu u)'=\frac{d}{dt}\bigl( \langle D^2u_t\nu_t,\nu_t\rangle \bigr)_{t=0}&=\langle D^2u'\nu,\nu\rangle-\langle D^2u\nabsig \th_\nu,\nu\rangle-\langle D^2u\nu,\nabsig \th_\nu\rangle
\end{align*}
As $u(x)=p(\abs{x})$ for a one dimensional function $p:[0,1]\to\R$, one can express $D^2u$:
\[
D^2u(x)=p''(\abs{x})\nu(x)\otimes\nu(x)+p'(|x|)D\nu(x).
\]
As $(\nu\otimes \nu)\cdot\nabsig \th_\nu=0$ and $D\nu^T=D\nu$ is tangential (\textit{i.e.} $D\nu\cdot\nu=0$, see \eqref{eq:nu'}), then we obtain $(\parnunu u)'=\parnunu u'$. Finally, using again $\nabsig u=0$, we have
\begin{align*}
    (\nabsig u\cdot\nabsig \th_\nu)'=\frac{d}{dt}(\nabsigt u_t\cdot \nabsigt \th^t_{\nu_t})_{t=0}&=\frac{d}{dt}(\nabsigt u_t)_{t=0}\cdot \nabsig \th_\nu\\
    &= \Big(\nabsig u'+(\nabsig u\cdot\nabsig \th_\nu)\nu+\parnu u\nabsig \th_\nu\Big)\cdot \nabsig \th_\nu\\
    &=\nabsig u'\cdot \nabsig \th_\nu +\parnu u|\nabsig \th_\nu|^2.
\end{align*}
Gathering the terms in $u,v$ the ones in $u', v'$ and the ones in $u'', v''$ together, one can express $\parnunu u''$ from \eqref{eq:2ndderiv_u'} as the sum of $Y_1:=Y_1(u'',v'')$, $Y_2:=Y_2(u',v')$ and $Y_3(u,v)$, where
\begin{align*}
    Y_1&=v''-u'',\\
    Y_2&=2\Big(\th_\nu\left(\parnu (v-u)'-\parnunu u'\right)+\nabsig u'\cdot\nabsig \th_\nu\Big),\\
    Y_3&=
    (\parnu (v-u)-\parnunu u)(\th_\nu'+\th_\nu\parnu \th_\nu)
    +\th_\nu^2(\parnunu (v-u)-\partial_{\nu\nu\nu} u)+\parnu u|\nabsig \th_\nu|^2.
\end{align*}
Multiplying the equation in \eqref{eq:2ndderiv_u'} by $u$ and integrating by parts one has
\begin{align*}
    -\int_\B u''\Delta u+\int_\S(u''\parnu u-\parnu u''u)=-\int_\B \Delta u'' u=\lambda''\int_\B u^2+\lambar\int_\B u''u
\end{align*}
so that 
\begin{align}
    \lambda''\int_\B u^2&=\int_\S(u''\parnu u-\parnu u''u)\nonumber\\
    &=\int_\S \Big(u''(v-u)-u(Y_1+Y_2+Y_3)\Big)\nonumber\\
    &=\int_\S\Big((u''v-uv'')-Y_2u-Y_3u\Big) \nonumber 
\end{align}
We integrate by parts the term $\int_\S u\nabsig u'\cdot\nabsig \th_\nu$ present in $\int_\S Y_2u$:
\begin{align*}
\int_\S u\nabsig u'\cdot\nabsig \th_\nu=\int_\S\nabsig u'\cdot\nabsig (u\th_\nu)&=-\int_\S u\th_\nu\Delta_\Sigma u'\\&=\int_\S u\th_\nu(\lambar u'+H\parnu u'+\parnunu u'),
\end{align*}
thus providing 
\[
-\int_\S Y_2u=-2\int_\S u\th_\nu\Big(\parnu v'+\lambar u'+(H-1)\parnu u'\Big).
\]
We get as a consequence
\begin{equation}
    \lambda''\int_\B u^2=\int_\S\Big((u''v-uv'')-2 u\th_\nu\Big(\parnu v'+\lambar u'+(H-1)\parnu u'\Big)-Y_3u\Big)\label{eq:lambda''_u^2}
\end{equation}

\textbf{Summing the two parts.} 

Summing \eqref{eq:lambda''_v2} and \eqref{eq:lambda''_u^2}, we get\begin{align*}
    \lambda''=\int_\S W_1+W_2+W_3
\end{align*}
where we have set 
\begin{align*}
    W_1&=(u''v-uv'')-(\lamti v''+u'')v\\
    W_2&=-2u\th_\nu\Big(\parnu v'+\lambar u'+(H-1)\parnu u'\Big)-2v\th_\nu\left(\parnu (\lamti v'+u')+H(\lamti v'+u')\right)\\
    W_3&=-Y_3u-v\th_\nu^2(\parnunu w+H\parnu w).
\end{align*}
We have
\[
W_1=-v''(\lamti v+u)=-v''\Delta_\Sigma v=0.
\]
As for $W_2$, using $w=0$ and $\parnu u =v-u$ on $\S$ one gets
\[
W_2=-2\th_\nu\Big(H\lamti vv'+(\lambar u+vH)u'+(\parnu u+Hu)\parnu u'\Big).
\]
 Now, by the boundary condition $\parnu u'=\th_\nu(\parnu (v-u)-\parnunu u)+(v'-u')$ from \eqref{eq:u_t'}, we get after simplifications
\[W_2=-2\th_\nu\Big(\parnu u v'+\left((H-1)\parnu u+\lambar u\right) u'+\th_\nu(\parnu (v-u)-\parnunu u)(\parnu u+Hu)\Big).\]
In $Y_3$,  \eqref{eq:1st_2nd_vol_preserv} allows to rewrite 
\[\int_\S u(\parnu (v-u)-\parnunu u)(\th_\nu'+\th_\nu\parnu \th_\nu)=-\int_\S H\th_\nu^2u(\parnu (v-u)-\parnunu u),\]
since $u(\parnu (v-u)-\parnunu u)$ is a constant over $\S$ (thanks to Corollary \ref{cor:critpt} and the chosen extension of $v$, see \eqref{eq:def_ext_egf}). 
We thus obtain\begin{equation}\nonumber
\lambda''=\int_\S \alpha \th_\nu^2-\int_\S u\parnu u|\nabsig \th_\nu|^2-2\int_\S \th_\nu\Big((\parnu u) v'+\left((H-1)\parnu u+\lambar u\right) u'\Big),
\end{equation}
where 
\begin{align*}
\alpha:&= Hu(\parnu (v-u)-\parnunu u)-u(\parnunu (v-u)-\partial_{\nu\nu\nu} u)\\
&-v(\parnunu w+H\parnu w)-2(\parnu (v-u)-\parnunu u)(\parnu u+Hu).
\end{align*}
As a consequence, there only remains to simplify the expression of $\alpha$ to conclude the proof of the proposition. 

\textbf{Simplification of $\alpha$.} 

Computing separately the terms in $\partial_{\nu\nu\nu}u$, $\parnunu u$, $\parnu u$, and since $\parnu v=\parnunu v=0$ (thanks to the choice of extension of $v$, see \eqref{eq:def_ext_egf}) we get 
\begin{equation}\label{eq:alpha_not_simple}
\alpha=-(H-2)(\parnu u)^2+(\parnu u+Hu)(\parnunu u)+u(\partial_{\nu\nu\nu}u).
\end{equation}
Let us now find expressions for $\parnunu u$ and $\partial_{\nu\nu\nu}u$.

As $u$ is radial one we write  $u(x)=p(|x|)$ for some  $p\in\C^{2,\eta}(\R)$. From the decomposition of the Laplacian
\[\Delta u =\Delta_{\Sigma_r} u+H_r\partial_{\nu_r} u+\partial_{\nu_r\nu_r} u=H_r\partial_{\nu_r} u+\partial_{\nu_r\nu_r} u \ \text{ over } \mathbb{S}^{d-1}_r,\] 
for each $r\in (0,1]$, we thus get
\[\forall r\in (0,1), \
p''(r)+\frac{d-1}{r}p'(r)=-\lambar p(r),
\]
so that
\[\forall r\in(0,1),\ 
p'''(r)=-\frac{d-1}{r}p''(r)+\left(\frac{d-1}{r^2}-\lambar\right)p'(r),
\]
which gives for $r=1$, written in terms of $u$,
 \[
 \partial_{\nu\nu\nu}u=-H\parnunu u+(H-\lambar)\parnu u \ \text{ over } \S.
 \]
 On the other hand note that the decomposition of the Laplacian also ensures 
 \[
 \parnunu u=-H\parnu u+\Delta u=-H\parnu u -\lambar u \ \text{ over }\S.
 \]
 Using these two expressions, we rewrite \eqref{eq:alpha_not_simple}:
\begin{align*}
    \alpha&=-(H-2)(\parnu u)^2+(\parnu u+Hu)(\parnunu u)+u(-H\parnunu u+(H-\lambar)\parnu u)\\
    &=\parnu u\left(-(H-2)\parnu u+\parnunu u+u(H-\lambar)\right)\\
    &=\parnu u\left(-(H-2)\parnu u-\lambar u-H\parnu u+u(H-\lambar)\right)\\
    &=\parnu u\left(-2(H-1)\parnu u+u(H-2\lambar) \right)
\end{align*}
This finishes the proof.
\end{proof}

When analysing the second-shape derivative of a functional at a critical point (in our case the ball), it usually reveals very convenient to work with a diagonalised expression (see for instance \cite{DKL16,M20}, among many others, and the general reference \cite{DL19}). Building on the expression given by Proposition \ref{prop:second_shaped}, in Proposition \ref{prop:diag_lag} below we thus provide a diagonalised form of the shape hessian of $\lambda$. Let us start by recalling some standard facts regarding spherical harmonics.

\textbf{Spherical harmonics.} As is very common for such energetic functionals, we will see that a diagonalisation basis is given by the $L^2(\S)$ basis of the spherical harmonics, which are defined as the restrictions to the sphere of homogeneous harmonic polynomials defined in $\R^d$ (see \cite{Mull} for a general reference on spherical harmonics). For any $k\in\N$, we thus denote by $\left(Y^k_1,\ldots Y^k_{L_k}\right)$ an $L^2(\S)$-orthonormal basis of the space of spherical harmonics of order $k$ (\text{i.e.} the restrictions to $\S$ of homogeneous harmonic polynomials of degree $k$), so that $\{Y^k_\ell , k\in\N, 1\leq \ell \leq L_k\}$ is an orthonormal basis of $L^2(\S)$. Let us note that the family $\{Y^k_\ell \}_{k,\ell}$ is also diagonalising the Laplace-Beltrami operator $-\Delta_\Sigma$, as one has
\[
\forall k\in\N, \forall l\in\llbracket1,L_k\rrbracket,\ 
-\Delta_\Sigma Y^k_\ell =\sigma_kY^k_\ell .
\]
where the sequence of eigenvalues $\{\sigma_k\}_{k\in\N}$ is given by
\begin{equation}
    \label{eq:sigmak_harmonic}\forall k\in\N,\ \sigma_k=k(k+n-2).\end{equation}
If $(k,\ell), (k',{\ell'})$ are two integer couples, one has the integration by parts
\[
\int_{\S}\nabsig Y^k_\ell \cdot\nabsig Y^{k'}_{\ell'}=\sigma_k\int_{\S}Y^k_\ell Y^{k'}_{\ell'}=\begin{cases}
    \sigma_k, \text{ if } (k,\ell)=(k',{\ell'}),\\
    0, \text{ else.}
\end{cases}
\]
This entails that for any $h\in {W^{1,2}}(\S)$ written in the $\{Y^k_\ell \}_{k,\ell}$ basis as
\[
    h=\sum_{k=0}^{+\infty}\sum_{\ell=1}^{L_k}h_{k,\ell}Y^k_\ell ,
\]
for some real coefficients $h_{k,\ell}$,
one simply has
\begin{equation}
    \label{eq:L2_nab_h_ykl}
    \int_{\S}|h|^2=\sum_{k=0}^{+\infty}\sum_{\ell=1}^{L_k}h_{k,\ell}^2 \text{ and } 
    \int_{\S}|\nabsig h|^2=\sum_{k=0}^{+\infty}\sum_{\ell=1}^{L_k}\sigma_kh_{k,\ell}^2.
\end{equation}

\textbf{Diagonalisation of the shape Hessian.}
Any arbitrary $h\in\C^{2,\eta}(\S)$ is extended into a $\C^{2,\eta}(\R^d)$ function by setting
\begin{equation}
    \label{eq:ext_h}
\forall x\in\R^d,\ h(x):=\chi(x) h\left(\frac{x}{|x|}\right),
\end{equation}
for a smooth cut-off function $\chi$ constant in a neighbourhood of $\S$, thus getting an extension which is constant along the normal direction (close to $\S$). 
In the following proposition we write in a diagonalised form the second derivative of the Lagrangian $\Lm_\mu:=\lambda-\mu\mathrm{Vol}$ at the ball along an arbitrary $h\in\C^{2,\eta}(\S)$ (where $\mu$ is the Lagrange multiplier associated to the volume constraint, see Corollary \ref{cor:critpt}), which is defined as the second Fréchet derivative of $\Lm_\mu$ at the ball along the field $h\nu_\B$ (recall that $\nu_\B$ is extended through \eqref{eq:ext_nu}):
\begin{equation}
    \label{eq:notation_second_lmu}
\Lm_\mu''(\B)\cdot(h,h):=\Big(\frac{d^2}{dt^2}\Big(\Lm_\mu\left((\text{Id}+th\nu_\B)(\B)\right)\Big)_{t=0}.
\end{equation}

\begin{proposition}[Diagonalised second-shape derivative at the ball]\label{prop:diag_lag}
    Let $h\in \C^{2,\eta}(\S)$, which we suppose written in the spherical harmonics basis as
    \begin{equation}\label{eq:diag_h}
    h=\sum_{k=0}^{+\infty}\sum_{\ell=1}^{L_k}h_{k,\ell}Y^k_\ell ,
    \end{equation}
    and assume that $h_{0,1}=\int_\S h=0$.
    Let $\mu\in\R$ be the Lagrange multiplier associated to the volume constraint (see Corollary \ref{cor:critpt}) and $\mathcal{L}_\mu:=\lambda-\mu\mathrm{Vol}$ be the corresponding Lagrangian. Then
        \begin{equation}
        \label{eq:lmu_diag_exp}
    \Lm_\mu''(\B)\cdot(h,h)=\sum_{k=1}^{+\infty}\sum_{\ell=1}^{L_k} \left(\beta(\sigma_k-H)+\left(\frac{\gamma}{\sigma_k-\lamti}+\delta\right)\left(p_k+\parnu u\right)\right)h_{k,\ell}^2,\end{equation}
    where $(\beta,\gamma,\delta)$ are defined in Proposition \ref{prop:second_shaped}, we have set $p_k:=p_k(1)$,  $p_k:[0,1]\to\R$ being the unique solution to 
    \begin{equation}
        \label{eq:def_pk}
        \begin{cases}
-\frac{1}{r^{d-1}}\frac{d}{dr}\left(r^{d-1}\frac{dp_k}{dr}\right)+\left(\frac{\sigma_k}{r^2}-\lambar\right)p_k=0, &\text{in } (0,1),\\
p_k'(1)+d_kp_k(1)=-d_k\parnu u-\parnunu u,
\end{cases}
        \end{equation}
and $d_k$ denotes the positive number
\[d_k:=\left(1-\frac{1}{\sigma_k-\lamti}\right).\]
\end{proposition}
\begin{remark}
Since $\sigma_1=H$ while, as we will see below (see Lemma \ref{Le:Comparison}) $p_1=-\parnu u$, the coefficient of degree $k=1$ in  \eqref{eq:lmu_diag_exp} cancels. 
\end{remark}

\begin{proof}

\textbf{Step 1: expression of $\Lm_\mu''(\B)\cdot(h,h)$.} 

Relying on Proposition \ref{prop:second_shaped}, we first find an expression for $\Lm_\mu''(\B)\cdot(h,h)$.  
This is obtained by building a second-order preserving path $(X_t)$, in the fashion of \cite[Proof of Lemma 2.10]{DL19}. We set 
\[
X_t:=\text{Id}+th\nu_\B+\frac{t^2}{2}\xi
\]
for some field $\xi:\R^d\to\R^d$ which we determine now. 
As $\int_\Sigma h=0$, one always has thanks to the Hadamard  formula \eqref{eq:had_domain}
\[
\frac{d}{dt}\Big(\mathrm{Vol}(X_t(\B))\Big)_{t=0}=0,
\]
while (see the Hadamard  formula \eqref{eq:had_boundary})
\begin{align}
\frac{d^2}{dt^2}\Big(\mathrm{Vol}(X_t(B))\Big)_{t=0}&=\int_\S\left(\frac{d}{dt}\left(X_t'\cdot\nu_t\right)_{t=0}+h\left(\parnu h+Hh\right)\right)\nonumber\\
&\nonumber = \int_\S\left(\xi\cdot\nu+Hh^2\right)\\
&= \mathrm{Vol}'(\B)\cdot(\xi)+\mathrm{Vol}''(\B)\cdot(h\nu_\B,h\nu_\B), \label{eq:second_vol_Xt}
\end{align}
since $X_0'\cdot\frac{d}{dt}\left(\nu_t\right)_{t=0}=-h\nu_\B\cdot\nabsig h=0$ (see \eqref{eq:nu'}) and $\parnu h=0$ over $\S$.
Assuming that the field $\xi:\R^d\to\R^d$ is any field verifying the condition
\begin{equation*}
\int_\S \xi\cdot\nu=-\int_\S Hh^2,
\end{equation*}
we  get that $t\mapsto X_t$ is second-order volume preserving, \textit{i.e.}
\begin{equation}
    \label{eq:volpreserv_build}
\frac{d}{dt}\Big(\mathrm{Vol}(X_t(\B))\Big)_{t=0}=\frac{d^2}{dt^2}\Big(\mathrm{Vol}(X_t(\B))\Big)_{t=0}=0.
\end{equation}
Now, by the chain rule, and recalling that $\lambda'(\B)=\mu\mathrm{Vol}'(\B)$ (see Corollary \ref{cor:critpt}),
\begin{align*}
\frac{d^2}{dt^2}\Big(\lambda(X_t(\B))\Big)_{t=0}&=\lambda''(\B)\cdot(h\nu_\B,h\nu_\B)+\lambda'(\B)\cdot(\xi)\\
&=\lambda''(\B)\cdot(h\nu_\B,h\nu_\B)+\mu \mathrm{Vol}'(\B)\cdot(\xi)\\
&= \lambda''(\B)\cdot(h\nu_\B,h\nu_\B)-\mu\mathrm{Vol}''(\B)\cdot(h\nu_\B,h\nu_\B), 
\end{align*}
using \eqref {eq:second_vol_Xt} and \eqref{eq:volpreserv_build}. As a consequence we thus get
\begin{equation}
    \label{eq:exp_lambda_mu}
    \Lm_\mu''(\B)\cdot(h,h)=\frac{d^2}{dt^2}\Big(\lambda(X_t(\B))\Big)_{t=0}
\end{equation}
for some second-order volume preserving path $t\mapsto X_t$ verifying $(X_0'\cdot\nu)_{|\S}=h$. Applying Proposition \ref{prop:second_shaped}, we obtain
\begin{equation}
    \label{eq:lmu_hh_decomp}
\Lm_\mu''(\B)\cdot(h,h)=\int_{\S}\alpha h^2+\int_{\S}\beta|\nabla_\Sigma h|^2+\int_{\S} h\Big(\gamma v_h'+\delta u_h'\Big),
\end{equation}
where we recall (see \eqref{eq:syst_u'} and \eqref{eq:v'eq} and note that by Corollary \ref{cor:critpt} $u$ is radial while $v$ is constant) that $(u_h',v_h')=\frac{d}{dt}(u_t,v_t)_{t=0}$ are the solutions to 
\begin{equation}
    \begin{cases}
        -\Delta u_h'=\lambar u_h' &\text{in } \B,\\
        \parnu u_h'+u_h'=-h(\parnu u+\parnunu u)+v_h'&\text{over }\S,\\
        -\Delta_\Sigma v_h'=\lamti v_h'+u_h'+h \parnu u &\text{over }\S,\\
        \int_{\B}u_h'u+\int_{\S}v_h'v=0.
    \end{cases}\label{eq:syst_u'v'}
    \end{equation}
Let us first compute the ``elementary" elements of this quadratic form:
\[
a_{k,\ell}:=\int_{\S}\left(\alpha |Y^k_\ell |^2+\beta|\nabla_\Sigma Y^k_\ell |^2+ Y^k_\ell (\gamma v_{Y^k_\ell }'+\delta u_{Y^k_\ell }')\right).
\]

\textbf{Step
2: Computation of $a_{k,\ell}$.} 
Since 
\[
\int_\S |Y^k_\ell |^2=1 \text{ and } \int_\S|\nabla_\Sigma Y^k_\ell |^2=\sigma_k,
\]
one has
\begin{equation}
    \label{eq:akl_exp}
a_{k,\ell}=\alpha+\beta\sigma_k+\int_\S Y^k_\ell (\gamma v_{Y^k_\ell }'+\delta u_{Y^k_\ell }').
\end{equation}
We search for $u_{Y^k_\ell }'$ and $v_{Y^k_\ell }'$ in the form
    \begin{equation}
    \begin{cases}
        u_{Y^k_\ell }'(x)=p_k^l(|x|)Y^k_\ell (x/|x|) &\text{in } \B,\\
        v_{Y^k_\ell }'(x)=q_k^lY^k_\ell (x) &\text{over } \S,
    \end{cases}\label{eq:u'v'_rad}
    \end{equation}
    for some function 
    $p_k^l:[0,1]\to\R$ and number $q_k^l\in \R$ to be determined. Thanks to \eqref{eq:syst_u'v'} this rewrites
\[\begin{cases}
    
Y^k_\ell \left(-\frac{1}{r^{d-1}}\frac{d}{dr}\left(r^{d-1}\frac{dp_k^l}{dr}\right)+\frac{\sigma_k}{r^2}p_k^l\right)=\lambar p_k^l Y^k_\ell  &\text{in } \B,\\
Y^k_\ell \left((p_k^l)'(1)+p_k^l(1)\right)=Y^k_\ell \left(-\parnu u-\parnunu u+q_k^l\right) &\text{over } \S,\\
\sigma_kq_k^lY^k_\ell =Y^k_\ell \left(\lamti q_k^l+p_k^l(1)+\parnu u\right) &\text{over } \S,
\end{cases}
\]
while the fourth orthogonality condition is always verified since $u$ is radial, $v$ is constant (see Corollary \ref{cor:critpt}) and $\int_\S Y^k_\ell =0$:
\[
\int_\B \Big(p_k^l(|x|)Y^k_\ell (x/|x|)\Big)u(x)=0,\
\int_\S \Big(q_k^lY^k_\ell (x)\Big)v(x)=0.
\]
Therefore, simplifying by $Y^k_\ell $ we deduce
\begin{equation}
    \begin{cases}
-\frac{1}{r^{d-1}}\frac{d}{dr}\left(r^{d-1}\frac{dp_k^l}{dr}\right)+\left(\frac{\sigma_k}{r^2}-\lambar\right)p_k^l=0 &\text{in } (0,1),\\
(p_k^l)'(1)+d_kp_k^l(1)=-d_k\parnu u-\parnunu u,\\
q_k^l=\frac{1}{\sigma_k-\lamti}\left(p_k^l(1)+\parnu u\right),
\end{cases}
\label{eq:syst_pk_qk}
\end{equation}
where we have set for $k\geq1$ the positive number
\[d_k:=\left(1-\frac{1}{\sigma_k-\lamti}\right).\]
Since both systems \eqref{eq:syst_u'v'} and \eqref{eq:syst_pk_qk} have a unique solution (this follows from the uniqueness of solutions to \eqref{eq:syst_u'v'}), we deduce that the couple $(u_{Y^k_\ell }',v_{Y^k_\ell }')$ has the corresponding form \eqref{eq:u'v'_rad}. As $p_k^l$ and $q_k^l$ do not depend on $l$ we will drop the index $l$ to write more simply $p_k$ and $q_k$.
Plugging the above expression into \eqref{eq:akl_exp} one deduces 
\[
a_{k,\ell}=\alpha+\beta\sigma_k+\gamma q_k+\delta p_k,\]
denoting more simply $p_k:=p_k(1)$.
Now, noting the identity $\alpha=-H\beta+\parnu u\delta$ and using the expression of $q_k$ from \eqref{eq:syst_pk_qk} we obtain
\begin{align}\nonumber
a_{k,\ell}&=-\beta H+\beta\sigma_k+\frac{\gamma}{\sigma_k-\lamti}(p_k+\parnu u)+\delta(p_k+\parnu u)\\
&=\beta(\sigma_k-H)+\left(\frac{\gamma}{\sigma_k-\lamti}+\delta\right)(p_k+\parnu u).\label{eq:lmu_ykl_f}
\end{align}
\textbf{Step 3: Conclusion} We now show that for an arbitrary $h\in\C^{2,\eta}(\S)$ it holds
\begin{equation}
    \label{eq:lmu_quadratic}
\Lm_\mu''(\B)\cdot(h,h)=\sum_{k=1}^{+\infty}\sum_{\ell=1}^{L_k}a_{k,\ell}h_{k,\ell}^2,
\end{equation}
which concludes the proof, in the view of \eqref{eq:lmu_ykl_f}.

Observe first that the solution $(u'_h,v'_h)$ to \eqref{eq:syst_u'v'} is linear in $h$, which is directly implied by the linearity of the system \eqref{eq:syst_u'v'} in $h$ combined with uniqueness of its solution. Note on the other hand that if $(k,\ell)$ and $(k',{\ell'})$ are two integer couples, then using the expressions \eqref{eq:u'v'_rad} it holds
\[
\int_\S Y^k_\ell v'_{Y^{k'}_{\ell'}}=q_k\int_\S Y^k_\ell Y^{k'}_{\ell'} \text{ and } \int_\S Y^k_\ell u'_{Y^{k'}_{\ell'}}=p_k\int_\S Y^k_\ell Y^{k'}_{\ell'}.
\]
Plugging these two observations in expression \eqref{eq:lmu_hh_decomp} and recalling \eqref{eq:L2_nab_h_ykl}, we obtain the desired \eqref{eq:lmu_quadratic}, thus finishing the proof of the proposition.
\end{proof}

\textbf{Second-order optimality conditions (Proof of Theorem \ref{thm:local_ball})}
Relying on Proposition \ref{prop:diag_lag}, we are now in a position to perform the local analysis announced in Theorem \ref{thm:local_ball}.
For the rest of this section we reintroduce the parameters $c_i$, $c_b$ in the notation, thus denoting $\lambda_{c_i,c_b}$ (instead of $\lambda$) the eigenvalue we seek to optimise, and $\Lm_{c_i,c_b,\mu}$ (instead of $\Lm_\mu$) the Lagrangian functional (defined in Proposition \ref{prop:diag_lag}).
Let us start by commenting on the gap to fill to derive a $\C^{2,\eta}$-local minimality statement from Theorem \ref{thm:local_ball}. 

\begin{remark}[From optimality conditions to local minimality]\label{rk:locmin}
    As we work in infinite dimensions, it is not direct (and in general it is false) that criticality and positivity of the second-order derivative implies local minimality. This problem comes from the fact that the differentiability norm (in our case $\C^{2,\eta}(\S)$) is much stronger than the coercivity norm (here ${W^{1,2}}(\S)$). A standard strategy for filling this so-called ``norm discrepancy" is to control the third-order terms appearing in the Taylor expansion of the functional by
    \[
    \|h\|_{{W^{1,2}}(\S)}^2\|h\|_{\C^{2,\eta}(\S)},
    \]
    where $h:\S\to\R$ is the perturbation (for a systematic approach, see \cite{DL19}). Although we strongly believe that this should hold, showing it rigorously would bring a lot of additional technical difficulties. In fact, it would imply to compute the second derivative of the eigenvalue $\lambda_{c_i,c_b}$ at a perturbation of the ball $(\text{Id}+th\nu_\B)(\B)$, and to estimate its variation with respect to the second derivative at the ball.
\end{remark}

We begin by showing two preparatory lemmas, which we will use in the course of the proof. The first one refines the estimate $\lambda_{c_i,c_b}-c_i>0$ from Proposition \ref{prop:lambda_cicb_est} in the case $c_i<-c_b$, by showing that the latter remains positive away from $0$ as $c_i\to-\infty$.
\begin{lemma}
    \label{le:cvg_lambar}
     Assume $c_b=0$. Then there exists $l_\infty>0$ such that
\begin{equation}
    \label{eq:lambda+Ci}
    \lambda_{c_i,0}(\B)-c_i\to l_\infty>0 \text{ as }c_i\to-\infty.
\end{equation}
\end{lemma}
\begin{proof}
   We denote more simply $l_{c_i}:=\lambda_{c_i,0}(\B)-c_i$.
First note that 
\[l_{c_i}=\min_{\substack{(U,V)\in {W^{1,2}}(\B)\times {W^{1,2}}(\S) \\ (U,V)\ne(0,0)}}\frac{\int_\B |\nabla U|^2+\int_\S |\nabsig V|^2+\int_\S (U-V)^2-c_i\int_\S V^2}{\int_\B U^2+\int_\S V^2}\]
is non-increasing in $c_i$. But since $c_i\leq0$, for any $(U,V)\in {W^{1,2}}(\B)\times {W^{1,2}}(\S)$ there holds
\[
\frac{\int_\B |\nabla U|^2+\int_\S |\nabsig V|^2+\int_\S (U-V)^2-c_i\int_\S V^2}{\int_\B U^2+\int_\S V^2}\geq 0
\]
with equality only if $U$ and $V$ are constants and $U_{|\S}=V=0$, which does not happen for the eigencouple $(U,V)=(u,v)$ since $u>0$ inside $\B$. In particular, we deduce $l_{-1}>0$.
Consequently, $l$ is non-increasing in $c_i$, and for each $c_i\leq-1$
\[
0<l_{-1}\leq l_{c_i}.
\]
This ensures \eqref{eq:lambda+Ci} and thus finishes the proof.
\end{proof}

\begin{remark}\label{rk:robin_linfty}
    Pushing further the analysis one can show that the problem corresponding to $l_{c_i}$ converges to $l_\infty=\lambda_R$ where $\lambda_R$ is the Robin eigenvalue defined by
    \[
    \lambda_R:=\inf_{u\in {W^{1,2}}(\Om), u\ne0}\frac{\int_\Om |\n u|^2+\int_\Sigma u^2}{\int_\Om u^2}.
    \]
\end{remark}

The second lemma gives some bound for the numbers $p_k:=p_k(1)$; recall (see \eqref{eq:def_pk}) that $p_k$ solves 
    \begin{equation*}
    \begin{cases}
        -\frac{1}{r^{d-1}}\frac{d}{dr}\left(r^{d-1}\frac{dp_k}{dr}\right)+\left(\frac{\sigma_k}{r^2}-\bar\lambda\right)p_k=0, \text{ in } (0,1),\\
        p_k'(1)+d_kp_k(1)=-d_k\partial_\nu u-\partial_{\nu\nu} u,
    \end{cases}
    \end{equation*}
    with $d_k:=\left(1-\frac{1}{\sigma_k-\tilde\lambda}\right)>0$. Recall also from Corollary \ref{cor:critpt} that the eigenfunction $u:\B\to\R$ is radial and positive. With a slight abuse of notation we still write $u:[0,1]\to\R$ the corresponding one-dimensional function. 
    
    \begin{lemma}\label{Le:Comparison}
     For any $c_i,c_b\in\R$ we have $p_1=-\frac{d}{dr}u$. Assuming furthermore that $c_b=0$, then
\begin{itemize}
    \item There exists $C>0$ such that for any $c_i<0$ it holds
     \[p_1>0, \text{ and for all } k\in \N^*, \ -C\leq p_k\leq p_1 \text{ in } (0,1).\]
     \item For any $c_i>0$, there exists $C=C(c_i)>0$ such that 
     \[
    \text{for all } k\in\N^*, \ |p_k|\leq C \text{ in }(0,1).\]
\end{itemize}
\end{lemma}
 
    \begin{proof}[Proof of Lemma \ref{Le:Comparison}]
    In the proof of the different points, we will use several times the computation of
    \[-\frac{1}{r^{d-1}}\frac{d}{dr}\left(r^{d-1}\frac{d}{dr}\left(\frac{a}{u}\right)\right)
    \]
    when $a:[0,1]\to\R$ is a given one-dimensional function (recall that $u>0$ in $(0,1)$). We now use a technique from \cite{Maz20} to study the second-order shape derivative. Denoting $b:=\frac{a}{u}$, one has in $(0,1)$:
    \begin{align}
    -\frac{1}{r^{d-1}}(r^{d-1}b')'&=-\frac{1}{r^{d-1}}\left(r^{d-1}\frac{a'}{u}\right)'+\frac{1}{r^{d-1}}\left(r^{d-1}\frac{u'}{u^2}a\right)'\nonumber\\
    &= -\frac{1}{r^{d-1}u}(r^{d-1}a')'+\frac{a'u'}{u^2}+\frac{1}{r^{d-1}}(r^{d-1}u')'\frac{a}{u^2}+u'\left(\frac{b}{u}\right)'\nonumber\\
    &=-\frac{1}{r^{d-1}u}(r^{d-1}a')'+\frac{1}{r^{d-1}}(r^{d-1}u')'\frac{a}{u^2}+\frac{u'}{u}\left(\frac{a'}{u}+b'-b\frac{u'}{u}\right)\nonumber\\
    &= -\frac{1}{r^{d-1}u}(r^{d-1}a')'+\frac{1}{r^{d-1}}(r^{d-1}u')'\frac{b}{u}+2\frac{u'b'}{u}\nonumber\\
    &= -\frac{1}{r^{d-1}u}(r^{d-1}a')'-\lambar b+2\frac{u'b'}{u}
    \label{eq:elliptic_a/u}
\end{align}
where we used \eqref{eq:principalegf_syst} in the last line to say that in radial coordinates the eigenfunction $u$ satisfies 
    \begin{equation}\label{eq:u_syst_rad}
    \begin{cases}
    -\frac{1}{r^{d-1}}\frac{d}{dr}\left(r^{d-1}\frac{du}{dr}\right)=\bar\lambda u\,, 
    \\ u'(1)+u(1)=v.
    \end{cases}    
    \end{equation}
    
    \textbf{Proof of $p_1=–\frac{d}{dr}u$.} 
    Letting $\varphi:=-\frac{du}{dr}$ we deduce by differentiating the equation in \eqref{eq:u_syst_rad} that $\varphi$ satisfies
    \[ -\frac{d^2\p}{dr^2}-\frac{d-1}{r}\frac{d\varphi}{dr}+\frac{d-1}{r^2}\varphi=\bar\lambda\varphi \text{ in }(0,1),\] which rewrites
    \[
    -\frac{1}{r^{d-1}}\frac{d}{dr}\left(r^{d-1}\frac{d\varphi}{dr}\right)+\left(\frac{\sigma_1}{r^2}-\lambar\right)\varphi=0.
    \] 
Also,
\[
\p'(1)+d_1\p(1)=-\frac{d^2}{dr^2}u(1)-d_1\frac{d}{dr}u=-d_1\parnu u-\parnunu u
\]
so that by uniqueness of the solution to \eqref{eq:def_pk} we have $\p=p_1$.

\textbf{Proof of $p_1\geq0$ when $c_i<0$.} Denoting $z_1:=\frac{p_1}{u}$, and using the computations \eqref{eq:elliptic_a/u} it holds
\begin{align*}
-\frac{1}{r^{d-1}}\left(r^{d-1}z_1'\right)'&= 
-\frac{1}{r^{d-1}u}(r^{d-1}p_1')'-\lambar z_1+2\frac{u'z_1'}{u}\\
&= -\frac{\sigma_1}{r^2}z_1+2\frac{u'z_1'}{u} \text{ on }(0,1). 
\end{align*}
Moreover, as $p_1=-\frac{du}{dr}$, it holds 
\[
p_1'(1)+Hp_1(1)=-\parnunu u-H\parnu u=-\Delta u=\lambar u>0,
\]
thanks to the decomposition of the Laplacian on $\S$ (see \eqref{eq:lap_tangential}), and the fact that $\lambar>0$ (see Proposition \ref{prop:lambda_cicb_est}). As
\[z_1'(1)+Hz_1(1)=\frac{1}{u(1)}\left(p_1'(1)+p_1(1)\right)-\frac{u'}{u}(1)z_1(1)
\]
this yields
\[
z_1'(1)+\left(H+\frac{u'}{u}(1)\right)z_1(1)=\frac{1}{u(1)}\left(p_1'(1)+Hp_1(1)\right)>0,
\]
where 
\[H+\frac{u'}{u}(1)=H+\frac{1+\lamti}{-\lamti}\geq\frac{1}{-\lamti}>0\]
where we also used Corollary \ref{cor:critpt} and Proposition \ref{prop:lambda_cicb_est}. 
This implies that $z_1\geq0$ in $[0,1]$: in fact, $z_1$ cannot reach an interior negative minimum at $r_0\in(0,1)$, as otherwise
\[
-\frac{1}{r^{d-1}}(r^{d-1}z_1')'(r_0)+\frac{\sigma_1}{r_0^2}z_1(r_0)<0, \text{ while } 2\frac{u'z_1'}{u}(r_0)=0.
\]
But $z_1$ cannot reach a negative minimum at the boundary, since this would imply
\[
0\geq z_1'(1)> -\left(H+\frac{u'}{u}(1)\right)z_1(1)>0.
\]
As a consequence, we get that $z_1\geq0$ in $(0,1)$, whence also $p_1\geq0$ since $u$ is positive.


\textbf{Proof of $p_k\leq p_1$ when $c_i<0$.} Introduce, for any $k> 1$, the function $s_k:=p_k-p_1$. We have that 
\begin{equation}
    \label{eq:ineqs_k}
    -\frac{1}{r^{d-1}}\frac{d}{dr}\left(r^{d-1}\frac{ds_k}{dr}\right)+\left(\frac{\sigma_k}{r^2}-\bar\lambda\right)s_k=\frac{\sigma_1-\sigma_k}{r^2}p_1\leq0\text{ in }(0,1),\end{equation} 
since $p_1\geq$. Furthermore, using that $p_1(1)=-\frac{d}{dr}u(1)=-\parnu u$ and the boundary condition for $p_k$, then 
\begin{align*}
    s_k'(1)+d_ks_k(1)&=d_kp_1(1)+p_1'(1)-(p_1'(1)+d_kp_1(1)),\\
    &=0\end{align*} so that, eventually, $s_k$ verifies 
\[\begin{cases}
-\frac{1}{r^{d-1}}\frac{d}{dr}\left(r^{d-1}\frac{ds_k}{dr}\right)+\left(\frac{\sigma_k}{r^2}-\bar\lambda\right)s_k\leq0&\text{ in }(0,1),
\\ s_k'(1)+d_ks_k(1)=0.
\end{cases}\]
We now show this implies that $s_k\leq0$ for each $k\in\N^*$. Setting
$w_k:=\frac{s_k}{u}$ in $(0,1)$, we have thanks to \eqref{eq:elliptic_a/u} and \eqref{eq:ineqs_k} that
\begin{equation}
    \nonumber
-\frac{1}{r^{d-1}}(r^{d-1}w_k')'\leq -\frac{\sigma_k}{r^2}w_k+2\frac{u'w_k'}{u} \text{ on }(0,1).
\end{equation}

On the other hand, using the boundary condition for $s_k$ we can write $w_k'(1)+d_kw_k(1)=-s_k(1)\frac{u'}{u^2}(1)$, so that
\[
w_k'(1)+\left(d_k+\frac{u'}{u}(1)\right)w_k(1)=0,
\]
where (using Corollary \ref{cor:critpt} and Proposition \ref{prop:lambda_cicb_est})
\begin{align*}
d_k+\frac{u'}{u}(1)&=1-\frac{1}{\sigma_k-\lamti}+\frac{1+\lamti}{-\lamti}\\
&=\frac{1}{-\lamti}\left(1-\frac{-\lamti}{\sigma_k-\lamti}\right)>0.
\end{align*} 
Working from the inequation verified by $w_k$ and its boundary condition, the same maximum principle argument as the one above for showing $p_1\geq$ leads to $w_k\leq0$, hence to $p_k\leq p_1$ in $(0,1)$.
\textbf{Proof of the bounds on $p_k$.}
Observe that 
\[
        -\frac{1}{r^{d-1}}\frac{d}{dr}\left(r^{d-1}\frac{dp_k}{dr}\right)+\left(\frac{\sigma_k}{r^2}-\bar\lambda\right)p_k=0\]
        
        while \[p_k'(1)+d_kp_k(1)\]  is bounded in $k\geq2$.

        Letting $j_0(r):=r^2$, it holds
        \[
       \left( -\frac{1}{r^{d-1}}\frac{d}{dr}\left(r^{d-1}\frac{d}{dr}\right)+\frac{\sigma_k}{r^2}-\bar\lambda\right)j_0=-2d+\sigma_k-\lambar r^2
        \]
        Noting that $\sigma_k\to+\infty$, if $c_i>0$ one can find $C>0$ and $k_0\in\N$ such that for any $k\geq k_0$
\[
\begin{cases}
    \left( -\frac{1}{r^{d-1}}\frac{d}{dr}\left(r^{d-1}\frac{d}{dr}\right)+\frac{\sigma_k}{r^2}-\bar\lambda\right)(p_k+Cj_0)=C(\sigma_k-2d-\lambar r^2)>0\text{ in } (0,1),\\
    (p_k+Cj_0)'(1)+d_k(p_k+Cj_0)(1)>0.
\end{cases}
\]
The same can be done for $c_i<0$ with in addition a constant $C>0$ independent of $c_i$, since thanks to Lemma \ref{le:cvg_lambar} we have $0<\lambar\leq l_\infty$ for any $c_i<0$.

We can thus perform the same maximum principle argument as before to deduce that $p_k+Cj_0\geq0$ in $(0,1)$ for $k\geq k_0$, thus getting the existence of some $C'>0$ (independent of $c_i$ in the case $c_i<0$) such that $p_k\geq -C'$ for all $k\geq1$. The upper bound for the $c_i>0$ case is proven likewise, by considering instead $p_k-Cj_0$. The conclusion follows.
\end{proof}

We are now ready to prove the main result of this section. 

\begin{proof}[Proof of Theorem \ref{thm:local_ball}] The fact that the ball is a critical point of the Lagrangian was proven in Corollary \ref{cor:critpt}. We are thus left with dealing with the second-order optimality condition for the two points, which we treat separately.

Note that we will only study the second-order optimality condition for the fields $\th=h\nu_\B$, where $h\in\C^{2,\eta}(\S)$. This is not restrictive with respect to a general field $\th\in\C^{2,\eta}(\R^d,\R^d)$, as the so-called structure theorem (see \cite[Theorem 5.9.2]{HP18}) yields for all $\th\in\C^{2,\eta}(\R^d,\R^d)$
\begin{equation}
    \label{eq:structure_2nd_d}
\Lm_\mu''(\B)\cdot(\th,\th)=
\frac{d^2}{dt^2}\Big(\Lm_\mu\left((\text{Id}+t\th)(\B)\right)\Big)_{t=0}=\Lm_\mu''(\B)\cdot(h,h)
\end{equation}
where $h:=(\th_\nu)_{|\S}\in \C^{2,\eta}(\S)$ (the notation $\Lm_\mu''(\B)\cdot(h,h)$ was introduced in \eqref{eq:notation_second_lmu}), since $\B$ is a critical point of $\Lm_\mu$ and $\Lm_\mu$ is twice Fréchet differentiable (see Proposition \ref{prop:diff_egf}).

\textbf{Proof of Theorem \ref{thm:local_ball}, item \eqref{it:notlocalmin}.}

 By considering $\lambda_{c_i,c_b}+c_b$ instead of $\lambda_{c_i,c_b}$, there is no loss of generality in assuming that $c_b:=0$. \BBB
    We are going to show that in this case one has
    \begin{equation}
        \label{eq:beta_neg}
        \beta<0,
    \end{equation}
    in the notations of Proposition \ref{prop:second_shaped}. Note that this is enough to get the conclusion: in fact, by Lemma \ref{Le:Comparison} the function $p_k$ is bounded uniformly in $k$, so that the sequence defined by
    \begin{equation*}
    \left(\frac{\gamma}{\sigma_k-\lamti}+\delta\right)\left(p_k+\parnu u\right)
    \end{equation*}
    is bounded in $k$. On the other hand, recalling \eqref{eq:sigmak_harmonic} we know that $\sigma_k\sim k^2$. These two pieces of information together with the diagonalised expression of the Lagrangian \eqref{eq:lmu_diag_exp} ensure the existence of a couple $(k,\ell)$  (with $k$ sufficiently large) such that 
    \[
    \int_\S Y^k_\ell =0 \text{ while }
    \Lm_{c_i,c_b,\mu}''(\B)\cdot(Y^k_\ell ,Y^k_\ell )<0.
    \]
    Setting $h:=Y^k_\ell $, this is enough to conclude.
    
    We are thus left with proving \eqref{eq:beta_neg}. We have 
    \[\beta=-u(\parnu u),\]
    with $u>0$ over $\S$, as it is a principal eigenfunction. 
    Thanks to Corollary \ref{cor:critpt}:
    \[\parnu u=v\lambda_{c_i,0}(\B),\] 
which is positive by Proposition \ref{prop:lambda_cicb_est}, thus
yielding the announced \eqref{eq:beta_neg}.

\textbf{Proof of Theorem \ref{thm:local_ball}, item \eqref{it:localmin}.}

Recall the expression of the diagonalised second-order shape derivative of the Lagrangian, given by \eqref{eq:lmu_diag_exp}:
\[\Lm_{c_i,c_b,\mu}''(\B)\cdot(h,h)=\sum_{k=1}^{+\infty}\sum_{\ell=1}^{L_k} \left(\beta(\sigma_k-H)+\left(\frac{\gamma}{\sigma_k-\lamti}+\delta\right)\left(p_k+\parnu u\right)\right)h_{k,\ell}^2.\]
Since $p_1+\parnu u=0$ by Lemma \ref{Le:Comparison} and moreover $\sigma_1=H$, the coefficient corresponding to the spherical harmonics of degree $1$ cancels, giving
\[\Lm_{c_i,c_b,\mu}''(\B)\cdot(h,h)=\sum_{k=2}^{+\infty}\sum_{\ell=1}^{L_k} \left(\beta(\sigma_k-H)+\left(\frac{\gamma}{\sigma_k-\lamti}+\delta\right)\left(p_k+\parnu u\right)\right)h_{k,\ell}^2.\]
We now prove that in the regime $c_i\ll -c_b$ then
\begin{equation}
    \label{eq:2dLag_neg}
\forall k\in\N\setminus\{0,1\},\ 
\beta(\sigma_k-1)+\left(\frac{\gamma}{\sigma_k-\lamti}+\delta\right)\left(p_k+\parnu u\right)>0,
\end{equation}
which is enough to yield the conclusion.

We thus assume that $c_i+c_b\to-\infty$.  Again, one can assume without loss of generality that $c_b:=0$ while $c_i\to-\infty$. 
Thanks to Lemma \ref{le:cvg_lambar}, we have
\begin{equation*}
    \lambda_{c_i}-c_i\to l_\infty>0 \text{ as }c_i\to-\infty,
\end{equation*}
for some $l_\infty>0$, by setting $\lambda_{c_i}:=\lambda_{c_i,0}(\B)$.
We prove \eqref{eq:2dLag_neg} by performing an asymptotic analysis of the coefficients. 
To start with, by Proposition \ref{prop:lambda_cicb_est} and Corollary \ref{cor:critpt}:
\[
\parnu u=v\lambda_{c_i}<0, u_{|\S}=v(1-\lambda_{c_i})>0
\]
so that recalling $\beta=-u\parnu u$,
\begin{align}\nonumber
\forall k\in\N\setminus\{0,1\},\ \beta(\sigma_k-H)&\geq v^2\lambda_{c_i}(\lambda_{c_i}-1)(\sigma_2-H)\\
&=v^2\lambda_{c_i}^2(d+1+o(1)).
\label{eq:beta_k2_pos}
\end{align}
 On the other hand, for each $k\geq2$,
\[
0\leq \frac{\gamma}{\sigma_k-\lamti}\leq\frac{\gamma}{\sigma_2-\lamti}
=\frac{-2v\lambda_{c_i}}{\sigma_2+1-v\lambda_{c_i}}\to2 \text{ when } c_i\to-\infty,
\]
using that $\lambda_{c_i}\to-\infty$ by \eqref{eq:lambda+Ci}. Moreover,
\begin{align*}
\delta&=-2v\left((H-1)\lambda_{c_i}+(\lambda_{c_i}-c_i)(1-\lambda_{c_i})\right)\\
&=-2v\lambda_{c_i}\left((d-2)-l_\infty+o(1)\right),
\end{align*}
 if $(d-2)-l_\infty\ne0$, using \eqref{eq:lambda+Ci} again. Finally, 
\[ \forall k\in\N\setminus\{0,1\},\ 
-C+v\lambda_{c_i}\leq p_k+\parnu u\leq p_1+v\lambda_{c_i},
\]
where the constant $C>0$ is taken from Lemma \ref{Le:Comparison}.
We thus have for all $k\in\N, k\geq2$ and $c_i\ll0$

\begin{align}\nonumber
    \left|\frac{\gamma}{\sigma_k-\lamti}+\delta\right|\left|p_k+\parnu u\right|&\leq\left(\frac{\gamma}{\sigma_2-\lamti}+|\delta|\right)|C'+v\lambda_{c_i}|\\
    &=2v^2\lambda_{c_i}^2|(d-2)-l_\infty+o(1)|\label{eq:est_gammadelta},
\end{align}
if  $(d-2)-l_\infty\ne0$. If $(d-2)-l_\infty=0$, we first find $\delta=-2vl_\infty(1+o(1))$ instead, leading to the alternative bound 
\begin{equation}
    \label{eq:est_gammadelta_altern}
\left|\frac{\gamma}{\sigma_k-\lamti}+\delta\right|\left|p_k+\parnu u\right|\leq Cv^2|\lambda_{c_i}|\end{equation} for a constant $C>0$.

As $2\leq d\leq5$, one has 
\[
2\left((d-2)-l_\infty\right)<d+1
\]
since $l_\infty>0$. Combining \eqref{eq:beta_k2_pos} and either \eqref{eq:est_gammadelta} or \eqref{eq:est_gammadelta_altern}, we thus get in any case that there exists $C>0$ such that if $c_i\leq-C$ then \eqref{eq:2dLag_neg} holds true, hence concluding the proof.
\end{proof}

\begin{remark}\label{rk:high_dim}
    The low dimension hypothesis is only used at the very end of the proof of the second point, namely for asserting that 
\[
2\left((d-2)-l_\infty\right)<d+1,
\]
which stops being the simple consequence of the positivity of $l_\infty$ in dimension $d\geq6$. 
\end{remark}

\bibliographystyle{abbrv}
\bibliography{biblio.bib}

\appendix
\section{Differentiability of the eigenelements}\label{Ap:Differentiability}

This part is devoted to differentiability of the eigenelements from \eqref{eq:system_egv}. Let us be more precise: letting, for a  $\C^{2,\eta}$ \BBB bounded open set $\Om$ and $\th\in\C^{2,\eta}(\R^d,\R^d)$ with $\|\th\|_{W^{1,\infty}(\R^d,\R^d)}<1$, $\phi_\th$ be the diffeomorphism $\phi_\th:=\text{Id}+\th$, we prove that the map
\begin{equation}
    \label{eq:diff_map}
\theta\in\C^{2,\eta}(\R^d,\R^d)\mapsto (u_\th\circ\phi_\th,v_\th\circ\phi_\th,\lambda_\th)\in\C^{2,\eta}(\ov{\Om})\times\C^{2,\eta}(\Sigma)\times\R
\end{equation}
is $\C^\infty$ around $\th=0$, where $(u_\th,v_\th,\lambda_\th)$ are the normalized eigenelements associated to $\Om_\th=(\text{Id}+\th)(\Om)$ (see \eqref{eq:principalegf_syst}), \textit{i.e.} they satisfy
\begin{equation*}
\begin{cases}
    -\Delta u_\th=\lambar_\th u_\th & \text{in }\Omega_\th,\\
    -\Delta_{\Sigma_\th} v_\th=\lamti_\th v_\th+u_\th &\text{over }\Sigma_\th,\\
    \partial_{\nu_\th} u_\th+u_\th=v_\th &\text{over }\Sigma_\th,\\
    u_\th>0 \text{ in }\Om_\th,\ v_\th>0 \text{ over }\Sigma_\th,\\
     \int_{\Om_\th} u_\th^2+\int_{\Sigma_\th} v_\th^2=1,
\end{cases}
\end{equation*}
where $\lambar_\th:=\lambar(\Om_\th)$ and $\lamti_\th:=\lamti(\Om_\th)$ (themselves defined in \eqref{eq:def_lambdatilde_bar}).
The strategy of proof, based on an application of the implicit function theorem, is quite standard when it comes to proving differentiability of eigenelements. Let us refer to \cite[Theorem 2.2]{DK11} for a prototypical application of this argument. 

The statement is the folowing.

\begin{proposition}\label{prop:diff_egf_circ}
    Let $\Om$ be a  $\C^{2,\eta}$ \BBB bounded open set. Then the mapping defined in \eqref{eq:diff_map} is $\C^\infty$ in a neighborhood of $\th=0$.
\end{proposition}
\begin{proof} We keep on with the conventions $\Om:=\Om_0$, $\Sigma:=\Sigma_0$. 
Define the Banach space $\mathcal{B}$ 
\[
\mathcal{B}:=\left\{(u,v)\in \C^{2,\eta}(\ov{\Om})\times\C^{2,\eta}(\Sigma),\ \parnu u +u=v \text{ on } \Sigma\right\}.
\]
Let, for $\th\in \C^{2,\eta}(\R^d,\R^d)$, the mapping $F$ be defined on $\C^{2,\eta}(\R^d,\R^d)\times\mathcal{B}\times\R$ by
\[
F(\th,u,v,\lambda)=\left(-L_\th u-J_\th \lambar_\th u,-L_\th^\Sigma v-J_\th^\Sigma(\lamti v+u),\left(\int_\Om J_\th u^2+\int_\Sigma J_\th^\Sigma v^2\right)-1\right),
\]
where we have set 
\[
\begin{cases}
    J_\th:=\text{det}(D\phi_\th),\\
    A_\th:=J_\th D\phi_\th^{-1}D\phi_\th^{-T},
\end{cases}
\ \begin{cases}
    J_\th^\Sigma:=\text{det}(D\phi_\th)\left|D\phi_\th^{-T}\nu_\Sigma\right|,\\
    A_\th^\Sigma:=J_\th^\Sigma D\phi_\th^{-1}D\phi_\th^{-T},\end{cases}
\]
and $L_\th, L_\th^\Sigma$ are the operator  defined by
\[L_\th u:=\text{div}(A_\th\nabla u), \ \
 L_\th^\Sigma v:=(J_\th^\Sigma)^{-1}\text{div}_\Sigma\left(A_\th^\Sigma\nabla v-\frac{\langle A_\th^\Sigma\nabla v,\nu_\Sigma\rangle}{\langle A_\th\nu_\Sigma,\nu_\Sigma\rangle}A_\th\nu_\Sigma\right).
\]
Note that the mapping $F$ takes values in $\C^{0,\eta}(\ov{\Om})\times\C^{0,\eta}(\Sigma)\times\R$,  and that is is $\C^\infty$  when $\|\th\|_{W^{1,\infty}(\R^d,\R^d)}<1$.   \BBB
Thanks to \cite[Section 2]{DK11} and \cite[Lemma 3.3]{DKL16}, the operators $L_\th$ and $L_\th^\Sigma$ are the transports of the Laplacian and tangential Laplacian to $\Om$ and $\Sigma$, so that
\[
F(\th,u,v,\lambda)=0  \Longleftrightarrow (u,v,\lambda)=(u_\th\circ\phi_\th,v_\th\circ\phi_\th,\lambda_\th).
\]
As a consequence, if we show that the partial derivative 
\[
\partial_{(u,v,\lambda)}F(0,u_0,v_0,\lambda_0) \text{ is invertible},
\]
where $(u_0,v_0,\lambda_0)$ denote the eigenelements on $\Om_0=\Om$,
then the implicit function theorem implies that \eqref{eq:diff_map} is $\C^\infty$ around the origin.

First, observe that this derivative is given for any $(\hat{u},\hat{v},\hat{\lambda})\in\mathcal{B}\times\R$ by
\begin{align}\label{eq:partial_invert}
\partial_{(u,v,\lambda)}F(0,u_0,v_0,\lambda_0)\cdot(\hat{u},\hat{v},\hat{\lambda})=\left(-\Delta \hat{u}-\lambar_0\hat{u}-\hat{\lambda}u_0,-\Delta_\Sigma \hat{v}-(\lamti_0 \hat{v}+\hat{u})-\hat{\lambda}v_0,\right.\\
\left.2\left(\int_\Om u_0\hat{u}+\int_\Sigma v_0\hat{v}\right)\right).\nonumber
\end{align}
Let $T:\C^{0,\eta}(\ov{\Om})\times\C^{0,\eta}(\Sigma)\to \mathcal{B}$ be defined by $T(f,g)=(u,v)$ where $(u,v)$ is the only solution to 
\[
\begin{cases}
    -\Delta u+c_iu=f&\text{in }\Om,\\
    \parnu u +u=v &\text{over }\Sigma,\\
    -\Delta_\Sigma v+(1-c_b)v-u=g&\text{over }\Sigma.
\end{cases}
\]
\BBB
Then $T$ is well-defined, which can be seen through existence and uniqueness of a minimiser $(u,v)\in W^{1,2}(\Om)\times W^{1,2}(\Sigma)$ to the associated variational problem
\begin{align*}
\min_{(u,v)\in W^{1,2}(\Om)\times W^{1,2}(\Sigma)}\left\{\frac{1}{2}\int_\Omega\Big(|\nabla u|^2+c_i u^2\Big)+\frac{1}{2}\int_\Sigma \Big(|\nabla_\Sigma v|^2+(u-v)^2-c_b v^2\Big)\right.\\
\left.-\int_\Om fu-\int_\Sigma gv\right\},
\end{align*}
combined with elliptic estimates on the above system. For the latter, let us simply say that from the initial regularity of $(u,v)$, one first recovers $W^{2,2}(\Om)$ regularity of $u$ by looking at the equation on $u$ (see \cite[Section 2.1]{G11}). This implies in turn $u\in L^{\ov{p}}(\Sigma)$ for some $\ov{p}>p$ thanks to trace estimates and Sobolev embeddings, thus ensuring $v\in W^{2,\ov{p}}(\Sigma)$ by looking at the equation on $v$, hence again $u\in W^{2,\ov{p}}(\Om)$. Iterating this procedure we obtain in a finite number of steps $u\in\C^{0,\eta}(\Sigma)$, whence $(u,v)\in\mathcal{B}$ by Schauder estimates.

Moreover, by compactness of the inclusion of $\C^{0,\eta}(\ov{\Om})\times\C^{0,\eta}(\Sigma)$ into $\C^{2,\eta}(\ov{\Om})\times\C^{2,\eta}(\Sigma)$ and Schauder estimates, $T$ is compact as seen as an operator from $E=\C^{0,\eta}(\ov{\Om})\times\C^{0,\eta}(\Sigma)$ to itself.

Let us now prove invertibility of \eqref{eq:partial_invert}.
Let $(x,y,\mu)\in \C^{0,\eta}(\ov{\Om})\times\C^{0,\eta}(\Sigma)\times\R$, and we thus search for $(\hat{u},\hat{v},\hat{\lambda})\in\mathcal{B}\times\R$ such that
\[
\begin{cases}
    -\Delta\hat{u}-\lambar_0\hat{u}=x+\hat{\lambda}u_0,&\text{in } \Om,\\
    \parnu \hat{u}+\hat{u}=\hat{v},&\text{over } \Sigma,\\
    -\Delta_\Sigma\hat{v}-(\lamti_0\hat{v}+\hat{u})=y+\hat{\lambda}v_0, &\text{over } \Sigma,\\
\end{cases}
\]
complemented with the condition $\int_\Om \hat{u}u_0+\int_\Sigma\hat{v}v_0=\mu/2.$
Since $T_{|E}$ is compact, from the Fredholm alternative we know that the above system has a solution exactly when $(x+\hat{\lambda}u_0,y+\hat{\lambda}v_0)$ is orthogonal to $(u_0,v_0)$, \textit{i.e.} if
\[
\langle (x+\hat{\lambda}u_0,y+\hat{\lambda}v_0),(u_0,v_0)\rangle=0.
\]
This rewrites $\hat{\lambda}=-\langle (x,y),(u_0,v_0)\rangle$ by using the normalization condition, thus determining $\hat{\lambda}$. Furthermore, in this case, as $\lambda_0$ is a simple eigenvalue, if $(\ov{u},\ov{v})$ denotes a particular solution of this system then all the solutions write $(\ov{u},\ov{v})+s(u_0,v_0)$ for $s\in \R$. But the last condition implies
\[
2s=\mu-2\left(\int_\Om \ov{u}u_0+\int_\Sigma\ov{v}v_0\right),
\]
thus determining completely $(\hat{u},\hat{v})$. We have therefore shown the invertibility of \eqref{eq:partial_invert}, hence the proposition.
\end{proof}
In the fashion of \cite[Lemma 5.3.3]{HP18}, by composition we immediately deduce from Proposition \ref{prop:diff_egf_circ} that the mapping 
\[
\theta\in\C^{2,\eta}(\R^d,\R^d)\mapsto (u_\th,v_\th,\lambda_\th)\in\C^{0,\eta}(\R^d)\times\C^{0,\eta}(\R^d)\times\R
\]
is of class $\C^2$ close to $\th=0$, where we have extended the eigenfunctions $u_\th:\Om\to\R$ and $v_\th:\Sigma\to\R$ as in \eqref{eq:def_ext_egf}, \textit{i.e.}
\begin{equation*}
\begin{cases}
    u_\th:=P^1_\Om(u_\th\circ\phi_\th)\circ\phi_\th^{-1}\in \C^{2,\eta}(\R^d),\\ v_\th:=P^2_\Om(v_\th\circ\phi_\th)\circ\phi_\th^{-1}\in \C^{2,\eta}(\R^d).
\end{cases}
\end{equation*}
Note that as the second derivative of this application involves the second derivatives of $u_\th(x)$ and $v_\th(x)$ in $x$, we have to consider $\C^{0,\eta}$ norms rather than $\C^{2,\eta}$ for the  space  this application takes values in  (analogously to \cite[Lemma 5.3.3]{HP18}).
\begin{proposition}\label{prop:diff_egf}
    Let $\Om$ be a  $\C^{2,\eta}$ \BBB bounded open set. Then the mapping
    \[
    \theta\in\C^{2,\eta}(\R^d,\R^d)\mapsto (u_\th,v_\th,\lambda_\th)\in\C^{0,\eta}(\R^d)\times\C^{0,\eta}(\R^d)\times\R
    \]
    is  $\C^2$ in a neighborhood of $\th=0$.
\end{proposition}

\end{document}